\documentclass[12pt]{article}
\usepackage{amsfonts}
\usepackage{amssymb}
\usepackage{mathrsfs}
\usepackage{srcltx}
\usepackage{xcolor}
\usepackage{amsmath}
\usepackage{amsthm}
\usepackage{amstext}
\usepackage{amsopn}

\usepackage{subfigure}
\usepackage{tikz}

\usepackage{graphicx}
\usepackage{bm}
\newtheorem{theorem}{Theorem}[section]
\newtheorem{lemma}[theorem]{Lemma}
\newtheorem{proposition}[theorem]{Proposition}

\theoremstyle{definition}

\theoremstyle{remark}
\newtheorem{remark}[theorem]{Remark}

\numberwithin{equation}{section}

\newcommand{\ba}{\begin{array}}
\newcommand{\ea}{\end{array}}
\newcommand{\f}{\frac}

\newcommand{\la}{\lambda}
\newcommand{\R}{{\mathbb R}}
\newcommand{\ds}{\displaystyle}

\begin{document}
\date{}
\title{ \bf\large{Global dynamics of a Lotka-Volterra competition patch model}\thanks{S. Chen is supported by National Natural Science Foundation of China (Nos. 12171117 and 11771109) and Shandong Provincial Natural Science Foundation of China (No. ZR2020YQ01),  J. Shi is supported by US-NSF grant DMS-1715651 and DMS-1853598, and Z. Shuai is supported by US-NSF grant DMS-1716445. 
}}
\author{
Shanshan Chen\footnote{Email: chenss@hit.edu.cn}\\[-1mm]
{\small Department of Mathematics, Harbin Institute of Technology}\\[-2mm]
{\small Weihai, Shandong 264209, P. R. China}\\[2mm]
Junping Shi\footnote{ Email: jxshix@wm.edu}\\[-1mm]
{\small Department of Mathematics, William \& Mary}\\[-2mm]
{\small Williamsburg, Virginia 23187-8795, USA}\\[2mm]
Zhisheng Shuai\footnote{Email: shuai@ucf.edu}\\[-1mm]
{\small Department of Mathematics, University of Central Florida}\\[-2mm]
{\small Orlando, Florida 32816, USA}\\[2mm]
Yixiang Wu\footnote{Corresponding Author, Email: yixiang.wu@mtsu.edu} \\[-1mm]
{\small Department of Mathematics, Middle Tennessee State University}\\[-2mm]
{\small Murfreesboro, Tennessee 37132, USA}
}

\maketitle

\newpage

\begin{abstract}
The global dynamics of the two-species Lotka-Volterra competition patch model with asymmetric dispersal is classified under the assumptions of weak competition and  the weighted digraph of the connection matrix is strongly connected and cycle-balanced. { We show} that in the long time, either the competition exclusion holds that one species becomes extinct, or the two species reach a coexistence equilibrium, and the outcome of the competition is determined by the strength of the inter-specific competition and the dispersal rates. Our main techniques in the proofs follow the theory of monotone dynamical systems and a graph-theoretic approach based on the Tree-Cycle identity.  \\[2mm]
\end{abstract}

\section{Introduction}

Spatial dispersal of organisms in a heterogeneous environment with uneven resource distribution and varying connectivity have long been recognized as key components of ecological interactions \cite{Cantrell1991,Cantrell1998,hanski1999habitat}. Various models { with multiple patches} (or metapopulation models) have been proposed to investigate the impact of the environmental heterogeneity and the connectivity of subregions on the population dynamics.
{ For example, metapopulation models have been used to study the effects of habitat fragmentation on biodiversity \cite{fahrig2003effects}, the connectivity of habitat patches on population survival and extinction \cite{fahrig1985habitat}, the impact of dispersal on species richness and species relative abundances \cite{mouquet2003community}, and the outcome of multi-species interactions such as competition \cite{tilman1994competition} and predation  \cite{holt1977predation}, etc. }
Typically a patch model { consists} of a system of ordinary differential equations with local dynamics in  each patch coupled with dispersal dynamics between patches. 

A prominent example of patch models is the Lokta-Volterra competition model on a set of discrete set of habitats \cite{hastings1982dynamics, levin1974dispersion, Takeuchi1996}. In this paper, we consider the following $n$-patch two-species Lotka-Volterra competition model:
\begin{equation}\label{main}
\begin{cases}
u_i'=\mu_u\ds\sum_{j=1}^{n}(a_{ij}u_j-a_{ji}u_i)+u_i(p_i-u_i-cv_i), &i=1,\dots,n, \; t>0,\\
v_i'=\mu_v\ds\sum_{j=1}^{n}(a_{ij}v_j-a_{ji}v_i)+v_i(q_i-bu_i-v_i),&i=1,\dots,n, \; t>0, \\
u(0)=u_0\ge(\not\equiv)\bm 0,\;v(0)=v_0\ge(\not\equiv)\bm0.
\end{cases}
\end{equation}
Here $u=(u_1,\dots,u_n)$ and $v=(v_1,\dots,v_n)$ represent the population densities of two competing species in $n$ patches, respectively; $n$ is an integer greater or equal to $2$; $p_i, q_i>0$ measure the intrinsic growth rates of species $u_i, v_i$ in patch $i$, respectively; $b, c>0$ are the inter-specific competition rates of the two species, and the intra-specific competition rates are rescaled to be $1$ in \eqref{main}; $\mu_u,\mu_v\ge0$ are the dispersal rates of the two species, respectively; and the matrix $A=(a_{ij})_{n\times n}$ describes the movement pattern between patches where $a_{ij}\ge 0$ ($i\not=j$) is the degree of movement from patch $j$ to patch $i$. In previous studies of \eqref{main} \cite{hastings1982dynamics, levin1974dispersion, Takeuchi1996}, it was often assumed that $a_{ij}=a_{ji}$ so the dispersal is symmetric, which is not necessarily assumed here.

For our purpose, let $L=(L_{ij})_{n\times n}$ denote the {\it connection matrix} for the model, where
\begin{equation}
L_{ij}=\begin{cases}
 a_{ij},&i\ne j,\\
-\ds\sum_{k\ne i} a_{ki},&i=j.
\end{cases}
\label{eq-connection}\end{equation}
Thus \eqref{main} can be re-written as
\begin{equation}\label{main1}
\begin{cases}
u_i'=\mu_u\ds\sum_{j=1}^{n}L_{ij}u_j+u_i(p_i-u_i-cv_i), &i=1,\dots,n, t>0,\\
v_i'=\mu_v\ds\sum_{j=1}^{n}L_{ij}v_j+v_i(q_i-bu_i-v_i),&i=1,\dots,n, t>0, \\
u(0)=u_0\ge(\not\equiv)\bm0,\;v(0)=v_0\ge(\not\equiv)\bm0.
\end{cases}
\end{equation}

The global dynamics of \eqref{main1} when $n=1$ is well known, and the parameter range of $bc\leq 1$ is often referred as the weak competition { regime}. 
Throughout out this paper, we assume
\begin{enumerate}
    \item[(A1)] $b>0$, $c>0$, and $0<bc\le 1$; $p_i, q_i>0$ for all $i=1, 2, ..., n$.
    \item[(A2)] The connection matrix $L$ as defined in \eqref{eq-connection} is irreducible.
\end{enumerate}
The assumption (A1) means that the competition between the two species is weak while (A2) means that the digraph $\mathcal{G}$ associated with $A$ (also { $L$}) is strongly connected. 
Our results also assume that
\begin{enumerate}
    \item[(A3)]   The weighted digraph $\mathcal{G}$ associated with $L$ is cycle-balanced.
\end{enumerate}
The definition of the weighted digraph $\mathcal{G}$ and cycle-balanced will be given in Section 2. 
We note that each of the following is a special case of (A3):
\begin{enumerate}
        \item[(A3a)] $L$ is symmetric;
        \item[(A3b)] $n=2$;
        \item[(A3c)] Every cycle of the weighted digraph $\mathcal{G}$ associated with $L$ has two vertices.
\end{enumerate}
{ The dynamics of \eqref{main1} when $L$ is symmetric has recently been considered in  \cite{Slavik2020}.}


Under the assumptions (A1)-(A3), { we show} that { any positive equilibrium (or coexistence equilibrium)} of \eqref{main1} is linearly stable except a special case when $bc=1$ (see Theorem \ref{theorem_stable}). Together with the theory of monotone dynamical systems \cite{hess,hsu1996competitive, smith2008monotone}, {we show that the following alternatives hold} for the global dynamics of \eqref{main1}: either there exists a unique positive coexistence equilibrium of \eqref{main1} that is globally asymptotically stable; or \eqref{main1} has no coexistence equilibrium  and one of the two semitrivial equilibria is globally asymptotically stable while the other one is unstable (see Theorem \ref{theorem-class}).

The results established here resemble the corresponding ones for the reaction-diffusion-advection Lokta-Volterra competition model on a continuous spatial domain:
\begin{equation}\label{RD}
\begin{cases}
u_t=\mu_u\Delta u-\alpha_u \nabla \cdot [u\nabla Q(x)]+u(p(x)-u-cv), &x\in\Omega,\; t>0,\\
v_t=\mu_v\Delta v-\alpha_v \nabla \cdot [v\nabla Q(x)]+v(q(x)-bu-v), &x\in\Omega,\; t>0,\\
\ds\mu_u\frac{\partial u}{\partial n}-\alpha_u u \frac{\partial Q(x)}{\partial n} =\mu_v\ds\frac{\partial v}{\partial n}-\alpha_v v \frac{\partial Q(x)}{\partial n}=0, &x\in\partial\Omega,\; t>0,\\
u(x, 0)=u_0\ge0,\;v(x, 0)=v_0\ge0, &x\in\partial\Omega.
\end{cases}
\end{equation}
Here $u(x, t)$ and $v(x, t)$ are the population densities of two competing species at location $x \in \Omega$ and time $t$ respectively; { $Q$ is used to describe the advection direction of the species; $\alpha_u$ and $\alpha_v$ are the advection rates}; the habitat $\Omega$ is a connected bounded smooth domain in $\R^N$ for $N\ge 1$; { $n$ is the outward unit normal vector of the boundary $\partial\Omega$}. The combined effect of dispersal rates and environmental heterogeneity on the global dynamics of \eqref{RD} have been studied extensively in recent years in, for example,  \cite{dockery1998evolution,he2013effects,he2016global,hutson2005convergence,hutson2003competing,hutson2003evolution,lam2012uniqueness,Lou,Ni2020}. For the diffusive case of $\alpha_u=\alpha_v=0$ (corresponding to symmetric $L$ in \eqref{main1}), the global dynamics of \eqref{RD} with weak competition $bc\leq 1$ has been completely classified in \cite{he2016global}; 
and a similar classification for the general case of $\alpha_u,\alpha_v\ne 0$ (corresponding to asymmetric $L$ in \eqref{main1}) was also achieved in \cite{zhou2018global}  under the assumption of $\mu_u/\alpha_u=\mu_v/\alpha_v>0$. We refer interested readers to the  review articles \cite{lam2020selected,lou2008some} and books \cite{cantrell2004spatial, hess, ni2011mathematics} for more results for \eqref{RD}. { Our intention to prove that every coexistence equilibrium if exists is linearly stable except for the case when $bc=1$ is motivated by \cite{he2016global,zhou2018global}, and our classification results are stated in a similar fashion as those in \cite{he2016global,zhou2018global}. Here, we utilize the graph-theoretic approach in \cite{li2010global} to bridge the gap between symmetry and asymmetry of the connection matrix $L$. We remark that the patch model \eqref{main1} may not be considered simply as a discretization of the reaction-diffusion equation models in \cite{he2016global,zhou2018global}, especially when $L$ is not symmetric. For example, the patches configured as a directed tree graph describing the dynamics of species in river networks (see e.g. Figure~\ref{fig2}(c) ) is within the scope of our considerations which seems not to be covered by reaction-diffusion models in \cite{he2016global,zhou2018global}.   }

The dynamics of \eqref{main1} with $n=2$ (two-patch model) with $b=c$ and $\mu_u=\mu_v=\mu$ has been studied in \cite{cheng2019coexistence,gourley2005two,lin2014global}, see also recent work \cite{Jiang2020BMB} for three-patch case.   Our results here state that when $n=2$, the global dynamics of the model {\eqref{main1}}  is completely determined by the local dynamics of the semitrivial equilibria. Moreover, we generalize the results in \cite{lin2014global} to the case $n>2$ with cycle-balanced condition, and prove that there is a globally asymptotically stable positive equilibrium when the diffusion rate is small while one semitrivial equilibrium out-competes the other one when the diffusion rate is large. { Our results also extend the ones in \cite{Slavik2020} which assumes $L$ to be symmetric (this work also considers the case when the competition is strong while we do not consider it here).}

Our paper is organized as follows. In Section 2, we present some preliminary results, and in Section 3, we classify the global dynamics of \eqref{RD} for the general case. In Section 4, we apply the general results to two special cases for more detailed dynamics: (1) (equal resource) $p_i=q_i$ for all $i=1, 2,...,n$; (2) (equal competitiveness and equal diffusivity)  $b=c=1$ and $\mu_u=\mu_v=\mu$.  { In Section 5, we give a conclusion. }

\section{Preliminaries}

{In this section,  we provide the necessary preliminary results, which are important to the proof of our main results, the stability of coexistence equilibrium and the global dynamics of \eqref{main1},  in Section \ref{section_main}.  In Section \ref{section_graph}, we present some results from matrix theory and graph theory, which are used to deal with the asymmetry of the connection matrix $L$ in the proof of Theorems \ref{theorem_stable}-\ref{theorem-class}.  In Section \ref{section_dynamical}, we state some well-known results about the single-species model and the strict monotone theory. In particular, the strict monotone theory results state that the global dynamics of the competition system \eqref{main1} is largely determined by the local dynamics of the steady states. }

\subsection{Matrices and graphs}\label{section_graph}

A vector $u=(u_1,\cdots,u_n)\gg\bm0$ ($u\ge\bm0$) means that every entry of $u$ is positive (nonnegative); a vector $u>\bm0$ if $u\ge \bm0$ and $u\neq \bm0$. Let $A=(a_{ij})_{n\times n}$ be an $n\times n$ matrix and let $\sigma(A)$ be the set of eigenvalues of $A$.
The {\it spectral bound} $s(A)$ of $A$ is defined as
$$
s(A)=\max\{{\rm Re} \lambda: \lambda\in\sigma(A)\}.
$$
The matrix $A$ is {\it reducible} if we may partition $\{1, 2, ..., n\}$ into two nonempty subsets $E$ and $F$ such that $a_{ij}=0$ for all $i\in E$ and $j\in F$. Otherwise $A$ is {\it irreducible}. 

A {\it weighted digraph} $\mathcal{G}=(V, E)$ associated with the matrix $A$ (denoted as $\mathcal{G}_A$ in short) consists of a set $V=\{1,2,\ldots,n\}$ of vertices and a set $E$ of arcs $(i,j)$ (i.e., directed edges from $i$ to $j$) with weight $a_{ji}$, where $(i,j) \in E$ if and only if $a_{ji}>0$, $i\not=j$. A digraph is \textit{strongly connected} if, for any ordered pair of distinct vertices $i,j$, there exists a directed path from $i$ to $j$. A weighted digraph $\mathcal{G}_A$ is strongly connected if and only if the weight matrix $A$ is irreducible \cite{berman1979nonnegative}. A list of distinct vertices $i_1, i_2, ..., i_k$ with $k\ge 2$ form a {\it directed cycle} if $(i_m, i_{m+1})\in E$ for all $m=1, 2, ..., k-1$ and $(i_k, i_1)\in E$.

A subdigraph $\mathcal{H}$ of $\mathcal{G}$ is \textit{spanning} if $\mathcal{H}$ and $\mathcal{G}$ have the same vertex set. The \textit{weight} of a subdigraph $\mathcal{H}$ is the product of the weights of all its arcs.
A connected subdigraph $\mathcal{T}$ of $\mathcal{G}$ is a {\it rooted out-tree} if it contains no directed cycle, and there is one vertex, called the root, that is not { a} terminal vertex of any arcs while each of the remaining vertices is a terminal vertex of exactly one arc. A subdigraph $\mathcal{Q}$  of $\mathcal{G}$ is {\it unicyclic} if it is a disjoint union of two or more rooted out-trees whose roots are connected to form a directed cycle. 
Every vertex of unicyclic digraph $\mathcal{Q}$ is a terminal vertex of exactly one arc, and thus a unicyclic digraph has also been called a {\it contra-functional digraph} \cite[page 201]{harary}.

A square matrix is called a row (column)  \emph{Laplacian} matrix if all the off-diagonal entries are nonpositive and the sum of each row (column) is zero. For \eqref{main1}, we associate it with a row Laplacian matrix: 
\begin{equation}\label{laplaces}\mathcal{L}
=\begin{pmatrix}
\sum\limits_{k\ne1} a_{1k} & -a_{12}& \cdots &  -a_{1n}\\
-a_{21}& \sum\limits_{k\ne2}a_{2k}& \cdots &  -a_{2n}\\
\vdots & \vdots & \ddots & \vdots\\
-a_{n1}& -a_{n2}& \cdots &  \sum\limits_{k\ne n}a_{nk}\\
\end{pmatrix}.
\end{equation}
Note that for the connection matrix $L$ defined in \eqref{eq-connection}, $-L$ is a column Laplacian matrix, and the off-diagonal entries of $L$ and $-\mathcal{L}$ are the same.
Let $\alpha_i\ge 0$ denote the cofactor of the $i$-th diagonal element of $\mathcal{L}$. Then 
\begin{equation}\label{eq-ev}
(\alpha_1, \alpha_2, \ldots, \alpha_n)
\quad \mbox{and} \quad (1,1,\ldots,1)^T
\end{equation} 
are the left and right eigenvectors of $\mathcal{L}$ corresponding to eigenvalue $0$, respectively. If $L$ is irreducible (equivalently, the digraph $\mathcal{G}$ associated with $A$ is strongly connected), then $\mathcal{L}$ is also irreducible and thus $\alpha_i>0$ for all $i$. We will also call $\mathcal{G}$ as the digraph associated with $\mathcal{L}$, and $\mathcal{L}$ as the Laplacian matrix of $\mathcal{G}$ throughout this paper. 

The following Tree-Cycle identity has been established in \cite[Theorem 2.2]{li2010global}. 

\begin{proposition}[Tree-Cycle Identity]
\label{TC} Let  $\mathcal{G}$ be a strongly connected weighted digraph and let $\mathcal{L}$ be the Laplacian matrix of $\mathcal{G}$ as defined in \eqref{laplaces}. Let $\alpha_i$ denote the cofactor of the $i$-th diagonal element of $\mathcal{L}$. Then the following identity holds for $x_i,x_j\in D\subset \mathbb{R}^N, 1\le i,j\le n$ and any family of functions $\{F_{ij}: D \times D \to \mathbb{R}\}_{1\leq i,j\leq n}$
\begin{equation}\label{eq-tc}
\sum_{i=1}^n\sum_{j\not= i, j=1}^n   \alpha_i L_{ij} F_{ij}(x_i, x_j)= \sum_{\mathcal{Q} \in \mathbb{Q}} w(\mathcal{Q}) \sum_{(s,r)\in E(\mathcal{C}_{\mathcal{Q}})} F_{sr}(x_s, x_r),
\end{equation} 
where $\mathbb{Q}$ is the set of all spanning unicyclic digraphs of $(\mathcal{G}, {A})$, $w(\mathcal{Q})>0$ is the weight of $\mathcal{Q}$ (the product of weights of all directed edges on $\mathcal{Q}$), and $\mathcal{C}_\mathcal{Q}$ denotes the directed cycle of $\mathcal{Q}$ with arc set $E(\mathcal{C}_{\mathcal{Q}})$.
\end{proposition}



We recall that a weighted digraph $\mathcal{G}$ is said to be {\it cycle-balanced} \cite[Section~3]{li2010global} if for any cycle $\mathcal{C}$ in  $\mathcal{G}$ it has a corresponding reversed cycle $-\mathcal{C}$ and $w(\mathcal{C})=w(-\mathcal{C})$. Here $-\mathcal{C}$, the reverse of $\mathcal{C}$, have the same vertices but edges with {reversed} direction as $\mathcal{C}$.

The following result illustrates how the Tree-Cycle Identity can be used to bridge the gap between symmetry and asymmetry, and will be used later to analyze the eigenvalue problems related to equilibrium stability.

\begin{theorem}\label{thm-sy}
Let  $\mathcal{G}$ be a strongly connected weighted digraph that is cycle-balanced, and let $\mathcal{L}$ be the Laplacian matrix of $\mathcal{G}$ as defined in \eqref{laplaces}. Let $\alpha_i$ denote the cofactor of the $i$-th diagonal element of $\mathcal{L}$. Assume that $x_i,x_j\in D\subset \mathbb{R}^N$ for all $1\le i,j\le n$ and   $\{F_{ij}: D \times D \to \mathbb{R}\}_{1\leq i,j\leq n}$  be a family of functions satisfying 
\begin{equation}\label{eq-ineq}
    F_{ij}(x_i,x_j) + F_{ji}(x_j ,x_i) \ge 0, \quad 1\le i, j \le n, \;j\ne i.
\end{equation}
Then the following holds
\begin{equation}\label{eq-ds}
    \sum_{i=1}^n \sum_{j\ne i, j=1}^n \alpha_i L_{ij} F_{ij}(x_i,x_j) \, = \sum_{i=1}^n \sum_{j\ne i, j=1}^n \alpha_i L_{ij} F_{ji}(x_j,x_i) \ge 0.
\end{equation}
In addition, the double sum in \eqref{eq-ds} equals 0 if and only if $F_{ij}(x_i,x_j)+F_{ji}(x_j,x_i)=0$ for all distinct $i,j$.
\end{theorem}
\begin{proof}
For any unicyclic digraph $\mathcal{Q}$ with a directed cycle $\mathcal{C}_\mathcal{Q}$, reversing the directions of all directed edges in $\mathcal{C}_\mathcal{Q}$ (keeping the directions of all other directed edges) yields another unicyclic digraph $\mathcal{Q}'$. Since $\mathcal{G}$ is cycle-balanced, $\mathcal{Q}'$ is well-defined and $w(\mathcal{Q}')=w(\mathcal{Q})$. Notice that $(s,r)\in E(\mathcal{C}_\mathcal{Q})$ iff $(r,s)\in E(\mathcal{C}_{\mathcal{Q}'})$, and $(s,r)\in E(\mathcal{Q})-E(\mathcal{C}_\mathcal{Q})$ iff $(s,r) \in E(\mathcal{Q}')-E(\mathcal{C}_{\mathcal{Q}'})$. Perform this process to all unicyclic digraph $\mathcal{Q}$ on the right hand side of the Tree-Cycle identity \eqref{eq-tc}, and we obtain
\begin{align*}
& \sum_{i=1}^n \sum_{j\ne i, j=1}^n \alpha_i L_{ij} F_{ij}(x_i,x_j)\\
=& \sum_{\mathcal{Q} \in \mathbb{Q}} w(\mathcal{Q}) \sum_{(s,r)\in E(\mathcal{C}_{\mathcal{Q}})} F_{sr}(x_s, x_r)\\
=& \frac{1}{2} \Big[ \sum_{\mathcal{Q} \in \mathbb{Q}} w(\mathcal{Q}) \sum_{(s,r)\in E(\mathcal{C}_{\mathcal{Q}})} F_{sr}(x_s, x_r) + \sum_{\mathcal{Q}' \in \mathbb{Q}} w(\mathcal{Q}') \sum_{(s,r)\in E(\mathcal{C}_{\mathcal{Q}'})} F_{sr}(x_s, x_r) \Big]\\
=& \frac{1}{2} \Big[ \sum_{\mathcal{Q} \in \mathbb{Q}} w(\mathcal{Q}) \sum_{(s,r)\in E(\mathcal{C}_{\mathcal{Q}})} F_{sr}(x_s, x_r) + \sum_{\mathcal{Q} \in \mathbb{Q}} w(\mathcal{Q}) \sum_{(r,s)\in E(\mathcal{C}_{\mathcal{Q}})} F_{sr}(x_s, x_r) \Big]\\
=& \frac{1}{2} \Big[ \sum_{\mathcal{Q} \in \mathbb{Q}} w(\mathcal{Q}) \sum_{(s,r)\in E(\mathcal{C}_{\mathcal{Q}})} \Big(F_{sr}(x_s, x_r) +  F_{rs}(x_r, x_s)\Big) \Big]\\
\ge & 0,
\end{align*}
where the last inequality follows from \eqref{eq-ineq}.
\end{proof}

Since $\mathcal{L}$ is irreducible, the null space of $\mathcal{L}$ is one-dimensional. As a consequence, for any positive left eigenvector $(e_1,e_2, \ldots, e_n)$ 
of $\mathcal{L}$ corresponding to eigenvalue $0$, there exists a constant $\eta>0$ such that $e_i=\eta \alpha_i$ for all $i$. Therefore, the coefficients $\alpha_i$ in Theorem~\ref{thm-sy} can be replaced by the coordinators $e_i$ of any positive left eigenvector of $\mathcal{L}$ corresponding to eigenvalue $0$.

Finally we show some necessary and sufficient conditions for a digraph to be cycle-balanced, and enclose the proof in Appendix. 

\begin{proposition}\label{tree}
Let  $\mathcal{G}$ be a strongly connected weighted digraph with $n$ vertices associated with $n\times n$ matrix $A$ (assuming $a_{ii}=0$ for $1\leq i\leq n$). 
\begin{enumerate}
    \item If $A$ is symmetric ($a_{ij}=a_{ji}$), then $\mathcal{G}$ is cycle-balanced; and if $\mathcal{G}$ is cycle-balanced, then $A$ is sign pattern symmetric, that is, $a_{ij}>0$ if and only if $a_{ji}>0$ for any $j\ne i$. 
    \item If  $\mathcal{G}$ is cycle-balanced, then ${\mathcal G}$ has at least $2(n-1)$ arcs.
   
    
    \item
    If every cycle of  $\mathcal{G}$  has exactly two vertices, then $\mathcal{G}$ is cycle-balanced; and every cycle of  $\mathcal{G}$  has exactly two vertices if and only if $A$ is sign pattern symmetric and ${\mathcal G}$ has exactly $2(n-1)$ arcs. In addition, these arcs form two spanning trees in $\mathcal{G}$ of the opposite direction: one rooted in-tree and one rooted out-tree. In particular, if $n=2$, then $\mathcal{G}$ is cycle-balanced.
    
    \item Suppose that $\mathcal{G}$  is a complete graph ($a_{ij}>0$ for any $i\ne j$) with at least $3$ vertices. Then $\mathcal{G}$ is cycle-balanced if and only if each $3$-cycle of $\mathcal{G}$ is balanced, that is, for any three distinct vertices $i,j,k$, we have $a_{ij}a_{jk}a_{ki}=a_{ik}a_{kj}a_{ji}$.
\end{enumerate}
\end{proposition}

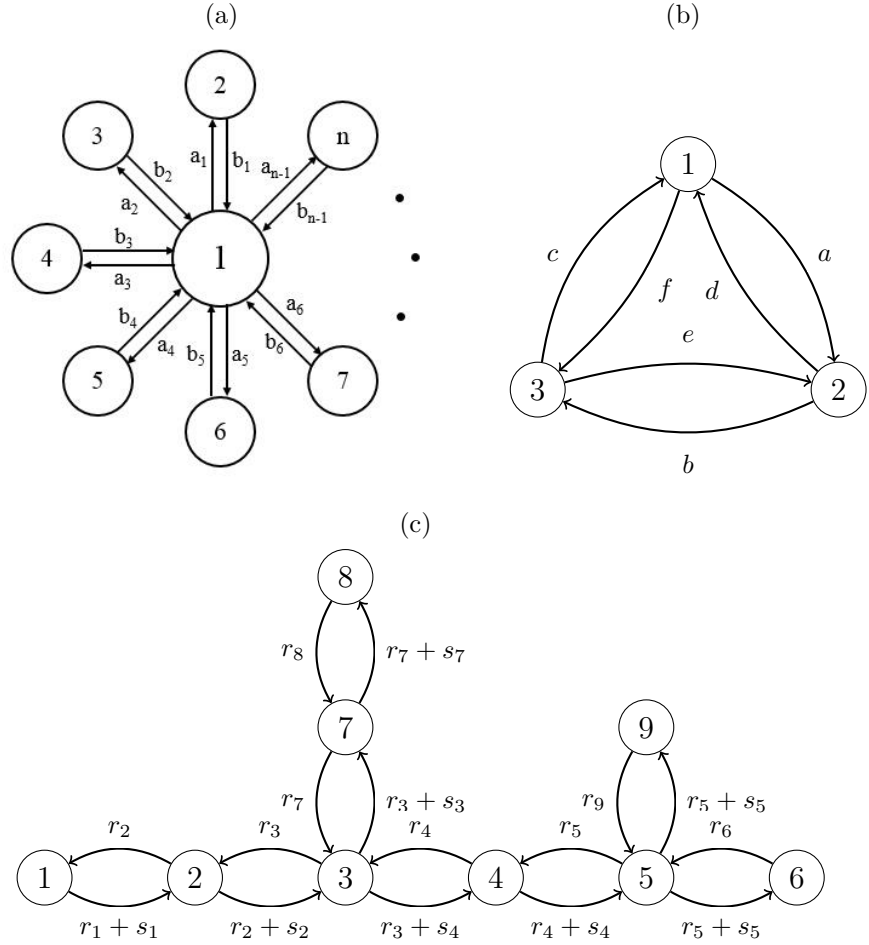
\begin{figure}[htbp]
\centering
\subfiguretopcaptrue
\subfigure[]{\includegraphics[width=0.45\textwidth]{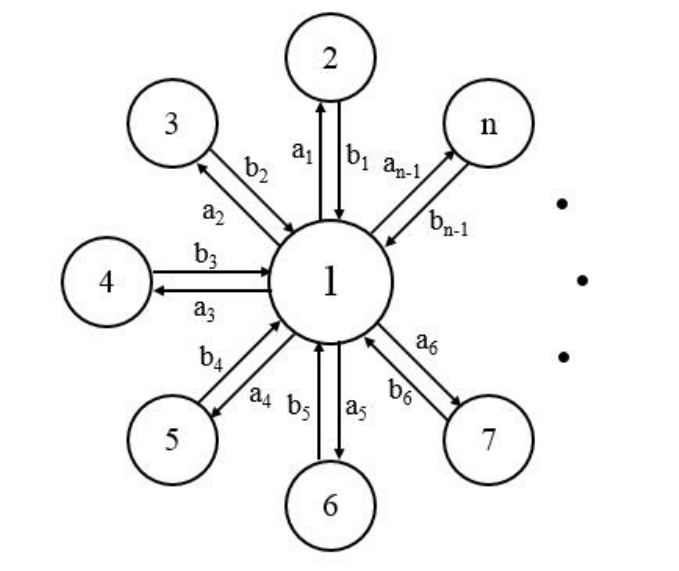}\label{fig1}}
\subfigure[]{
\begin{tikzpicture}
\begin{scope}[every node/.style={draw}, node distance= 2 cm]
    \node[white] (0) at (2,4.5) {};
    \node[circle] (1) at (2,3) {$1$};
    \node[circle] (2) at (4,0) {$2$};
    \node[circle] (3) at (0,0) {$3$};
\end{scope}
\begin{scope}[every node/.style={fill=white},
              every edge/.style={thick}]
    \draw[thick] [->](1) to [bend left=25] node[right=4] {{\footnotesize $a$}} (2);
    \draw[thick] [->](2) to [bend left=25] node[below=4] {{\footnotesize $b$}} (3);
    \draw[thick] [->](3) to [bend left=25] node[left=4] {{\footnotesize $c$}} (1);
    \draw[thick] [<-](1) to [bend right=15] node[left=5] {{\footnotesize $d$}} (2);
    \draw[thick] [<-](2) to [bend right=15] node[above=4] {{\footnotesize $e$}} (3);
    \draw[thick] [<-](3) to [bend right=15] node[right=5] {{\footnotesize $f$}} (1);
\end{scope}
\end{tikzpicture}
}\\
\subfigure[]{
 \begin{tikzpicture}
\begin{scope}[every node/.style={draw}, node distance= 1.5 cm]
    \node[circle] (1) at (0,0) {$1$};
    \node[circle] (2) at (2,0) {$2$};
    \node[circle] (3) at (4,0) {$3$};
    \node[circle] (4) at (6,0) {$4$};
    \node[circle] (5) at (8,0) {$5$};
    \node[circle] (6) at (10,0) {$6$};
    \node[circle] (7) at (4,2) {$7$};
    \node[circle] (8) at (4,4) {$8$};
     \node[circle] (9) at (8,2) {$9$};
\end{scope}
\begin{scope}[every node/.style={fill=white},
              every edge/.style={thick}]
    \draw[thick] [->](1) to [bend right] node[below=0.1] {{\footnotesize $r_1+s_1$}} (2);
    \draw[thick] [->](2) to [bend right] node[below=0.1] {{\footnotesize $r_2+s_2$}} (3);
    \draw[thick] [->](3) to [bend right] node[below=0.1] {{\footnotesize $r_3+s_4$}} (4);
    \draw[thick] [->](3) to [bend right] node[right=0.1] {{\footnotesize $r_3+s_3$}} (7);
    \draw[thick] [->](7) to [bend right] node[right=0.1] {{\footnotesize $r_7+s_7$}} (8);
    \draw[thick] [->](4) to [bend right] node[below=0.1] {{\footnotesize $r_4+s_4$}} (5);
    \draw[thick] [->](5) to [bend right] node[below=0.1] {{\footnotesize $r_5+s_5$}} (6);
    \draw[thick] [->](5) to [bend right] node[right=0.1] {{\footnotesize $r_5+s_5$}} (9);
    \draw[thick] [<-](1) to [bend left] node[above=0.1] {{\footnotesize $r_2$}} (2);
    \draw[thick] [<-](2) to [bend left] node[above=0.1] {{\footnotesize $r_3$}} (3);
    \draw[thick] [<-](3) to [bend left] node[above=0.1] {{\footnotesize $r_4$}} (4);
    \draw[thick] [<-](4) to [bend left] node[above=0.1] {{\footnotesize $r_5$}} (5);
    \draw[thick] [<-](5) to [bend left] node[above=0.1] {{\footnotesize $r_6$}} (6);
    \draw[thick] [->](7) to [bend right] node[left=0.1] {{\footnotesize $r_7$}} (3);
    \draw[thick] [->](8) to [bend right] node[left=0.1] {{\footnotesize $r_8$}} (7);
    \draw[thick] [->](9) to [bend right] node[left=0.1] {{\footnotesize $r_9$}} (5);
\end{scope}
\end{tikzpicture}
}
\caption{ { Three examples of strongly connected and cycle-balanced graphs. Each vertex in the figure represents a patch, and the weight of each directed edge is positive which corresponds to the degree of dispersal rate from the initial vertex to the terminal vertex. }(a) A star migration graph. (b) { A complete digraph with 3 edges}. It is cycle-balanced if $abc=def$. (c) A 
river network. }\label{fig2}
\end{figure}

The characterization in part 3 of Proposition \ref{tree} shows that a bi-directional tree is the cycle-balanced network with the minimal number of arcs. Examples of such bi-directional tree are the star graph in Figure~\ref{fig2}(a) and the river network in Figure~\ref{fig2}(c). Because of strong connectivity of $\mathcal{G}$, each row or column of $A$ has at least one non-zero entry. Also for these {networks}, the value of the weight $a_{ij}$ does not change the cycle-balanced property of the network, which is not the case for {networks} with longer cycles (like the one in Figure~\ref{fig2}(b) or complete graphs with more vertices as in part 4 of Proposition \ref{tree}).


\subsection{Dynamical systems}\label{section_dynamical}

{ In this section, we first state some results about the single species model, to which \eqref{main1} is reduced if $u=0$ or $v=0$. Then we present some results from monotone dynamics system theory. }

We recall results on a single species population model in a heterogeneous environment of $n$ patches ($n\ge 2$):
\begin{equation}\label{single}
w_i' = \mu \,\sum_{j=1}^n L_{ij}w_j+ w_i (r_i-w_i), \quad \quad i=1, \ldots, n,\; t>0,
\end{equation}
where $w_i$ denotes the population size (or density) in patch $i$. System \eqref{single} admits a trivial equilibrium $\mathbf{0}=(0,0,\ldots,0)$. { Suppose that $L$ is irreducible. Since the Jacobian matrix of the right hand side of \eqref{single} is cooperative (off-diagonal entries are nonnegative) and irreducible, the solutions of \eqref{single} have strongly monotonicity \cite[Theorem 4.1.1]{smith2008monotone}, i.e. for any two solutions $w^1(t)$ and $w^2(t)$ of \eqref{single}, if $w^1(0)>w^2(0)\ge \bm 0$ then $w^1(t)\gg w^2(t)$ for all $t>0$. In particular, if $w^1(0)>\bm 0$ then $w^1(t)\gg\bm 0$ for all $t>0$.} The global dynamics of \eqref{single} is as follows \cite{chen2019spectral,cosner1996variability,li2010global,Lu1993}, { which is essentially similar to the dynamics of the logistic model, i.e. the case $n=1$}: 
\begin{lemma}\label{lmm-single}
Suppose that the connection matrix $L$ as defined in \eqref{eq-connection} is irreducible, and $r_i>0$ for all $i=1, 2, ..., n$. Then the equilibrium $\mathbf{0}$ is unstable, and  \eqref{single} admits a unique positive equilibrium $w^*(\mu,r)=(w^*_1,\dots,w^*_n)$, which is globally asymptotically stable in $\mathbb{R}_+^n-\{\mathbf{0}\}$.
\end{lemma}

Throughout this paper, we assume that $p_i, q_i>0$ for all $i=1, 2, ..., n$.  By Lemma \ref{lmm-single}, \eqref{main1} has two semi-trivial equilibria, $E_1=(w^*(\mu_u,p), \mathbf{0})$ and $E_2=(\mathbf{0}, w^*(\mu_v,q))$, and one trivial equilibria $E_0=(\mathbf{0}, \mathbf{0})$. A positive equilibrium {(or coexistence equilibrium)} of \eqref{main1}, if exists, is denoted by $E=(u, v)$ when not causing confusion. 

An important tool to investigate the global dynamics of the Lotka-Volterra competition system \eqref{main1} is the  monotone dynamical system theory \cite{hess,hsu1996competitive,lam2016remark,smith2008monotone}. Let $X=\mathbb{R}_+^n\times\mathbb{R}_+^n$ equipped with an order $\le_K$ generated by the cone $K=\mathbb{R}_+^n\times\{-\mathbb{R}_+^n\}$. That is, for $x=(\bar u, \bar v), y=(\tilde u, \tilde v)\in X$, we say $x\le_K y$ if $\bar u\le \tilde u$ and $\bar v\ge \tilde v$; $x<_K y$ if $x\le_K y$ and $x\neq y$. { Then \eqref{main1} is cooperative with respect to $K$ \cite[Section 4.4]{smith2008monotone}, and the solutions of \eqref{main1} induce a strictly monotone dynamical system in $X$ in the sense that for any two initial data $(u_0^1, v_0^1)<_K(u_0^2, v_0^2)$, the corresponding solutions satisfy $(u^1(t), v^1(t))<_K(u^2(t), v^2(t))$ for all $t\ge 0$. Moreover since the Jacobian matrix of \eqref{main1} is irreducible in  $int(X)$, the dynamical system is strongly monotone in $int(X)$, i.e. for any two initial data $(u_0^1, v_0^1)<_K(u_0^2, v_0^2)$ with at least one of them in $int(X)$, the corresponding solutions satisfy $(u^1(t), v^1(t))\ll_K(u^2(t), v^2(t))$ for all $t> 0$ \cite[Theorem 4.1.1, Remark 4.1.1]{smith2008monotone}.

The following two results derived from the monotone dynamical system theory can be found in \cite{smith2008monotone}:
\begin{lemma}[{\cite[Theorem 4.4.2]{smith2008monotone}}]\label{M1}
Suppose that the connection matrix $L$ is irreducible and  $E_1$ is linearly unstable. Then one of the following holds:
\begin{itemize}
    \item[{\rm (i)}] $E_2$ attracts all solutions with initial data $(u_0,v_0)\in X$ satisfying $v_0>0$. In this case, $E_2$ is linearly stable or neutrally stable; 
    \item[{\rm (ii)}] There exists a positive equilibrium $E$ satisfying $E_2\ll_K E\ll_K E_1$ such that $E$ attracts all solutions with initial data $(u_0, v_0)\in X$ satisfying $E\le_K (u_0, v_0)<_K E_1$. 
\end{itemize}
\end{lemma}

\begin{lemma}[{\cite[Corollary 4.4.3]{smith2008monotone}}] \label{M2}
Suppose that the connection matrix $L$ is irreducible and  $E_1$ and $E_2$ are linearly unstable. Then there exist positive equilibria $E^1$ and $E^2$ of \eqref{main1} satisfying $E_2\ll_K E^2\le_K E^1\ll_K E_1$ such that the order interval
$$
[E^2, E^1]:=\{(u, v)\in X: \ E^2\le_K (u, v)\le_K E^1\}
$$
attracts all solutions with initial data $(u_0, v_0)$ satisfying $u_0>0$ and $v_0>0$. In particular, if $E^1=E^2$, then $E^1$ attracts all solutions as above.
\end{lemma}

The two Lemmas \ref{M1}-\ref{M2} are stated for the case $n=2$ in \cite{smith2008monotone}, but as remarked in  \cite[Page 70]{smith2008monotone}, they hold for any $n\ge 2$.  The following result can be derived by \cite[Theorem I.6.1]{hess}:
\begin{lemma}\label{M3}
Suppose that the connection matrix $L$ is irreducible and  $E^1$ and $E^2$ are  equilibria of \eqref{main1}  satisfying $E_2\le_K E^2<_K E^1\le_K E_1$. If $E^1$ and $E^2$ are stable and at least one of them is an isolated equilibrium, then there exists an unstable positive equilibrium $E\in [E^2, E^1]$. 
\end{lemma}

In view of previous lemmas, a key step as in \cite{he2016global, zhou2018global} is to show that  every positive equilibrium, if exists, is stable. }

\section{Stability of coexistence equilibrium}\label{section_main}


In this section we state our main results on the global dynamics of \eqref{main1}.
Recall that $w^*(\mu, r)$ is the unique positive steady state of \eqref{single}. Then the following result states that every positive equilibrium of \eqref{main1}, if exists, is locally asymptotically stable except for the degenerate case, when the competition is weak and the network $\mathcal{G}$ is cycle-balanced. 

\begin{theorem}\label{theorem_stable}
Suppose that (A1)-(A3) hold. A positive equilibrium $E=(u, v)$ of \eqref{main1}, if exists, is locally asymptotically stable except for the case $bc=1$ and $w^*(\mu_u, p)=cw^*(\mu_v, q)$ at which case $E$ is linearly neutrally stable.
\end{theorem}

\begin{proof}
Let $E=(u, v)$ be a positive equilibrium of \eqref{main1}. Then it satisfies the following equations:
\begin{equation}\label{ss}
\begin{cases}
\mu_u\ds\sum_{j=1}^{n}L_{ij}u_j+u_i(p_i-u_i-cv_i)=0, &i=1,\dots,n,\\
\mu_v\ds\sum_{j=1}^{n}L_{ij}v_j+v_i(q_i-bu_i-v_i)=0,&i=1,\dots,n.
\end{cases}
\end{equation}
Linearizing \eqref{main1} at $E$, we have the following eigenvalue problem
\begin{equation}\label{eig}
\begin{cases}
\mu_u\ds\sum_{j=1}^{n}L_{ij}\phi_j+(p_i-u_i-cv_i)\phi_i-u_i(\phi_i+c\psi_i)+\lambda\phi_i=0, &i=1,\dots,n,\\
\mu_v\ds\sum_{j=1}^{n}L_{ij}\psi_j+(q_i-bu_i-v_i)\psi_i-v_i(b\phi_i+\psi_i)+\lambda\psi_i=0,&i=1,\dots,n.
\end{cases}
\end{equation}
Here $(\phi, \psi)=((\phi_1, ..., \phi_n), (\psi_1, ..., \psi_n))$ is a principal eigenvector associated with the principal eigenvalue $\lambda$ of \eqref{eig}. We  normalize  $(\phi, \psi)$ such that $\phi_i>0$ and $\psi_i<0$ for all $i=1, ..., n$.  We will show that $\lambda\ge 0$, where the equality holds if and only if $bc=1$ and $w^*(\mu_u, p)=cw^*(\mu_v, q)$.

Multiplying the first equation of \eqref{eig} by $u_i$, the first equation of \eqref{ss} by $\phi_i$ and taking the difference, we have
\begin{equation}\label{diff}
    \mu_u \ds\sum_{j=1, j\neq i}^{n}L_{ij} (\phi_ju_i-\phi_i u_j)=u_i^2(\phi_i+c\psi_i)-\lambda u_i\phi_i, \ \ \ \ i=1,\dots,n.
\end{equation}
Let $(\alpha_1, ..., \alpha_n)$ be the left eigenvector of $\mathcal{L}$ as defined in \eqref{eq-ev}. Multiplying both sides of \eqref{diff} by $\alpha_i\phi_i^2/u_i^2$ and summing up all the equations, we obtain
\begin{equation}\label{diff1}
    \mu_u \ds\sum_{i=1}^{n} \ds\sum_{j=1, j\neq i}^{n}\alpha_i L_{ij} \left(\frac{\phi_i^2\phi_j}{u_i}-\frac{\phi_i^3u_j}{u_i^2}\right)=\ds\sum_{i=1}^{n}\alpha_i \phi_i^2(\phi_i+c\psi_i)-\lambda\ds\sum_{i=1}^{n}\alpha_i \frac{\phi_i^3}{u_i}.
\end{equation}

Let $\{F_{ij}:D \times D \to \mathbb{R}\}_{1\le i, j\le n}$ be a family of functions defined as
$\ds F_{ij}(x_i,x_j)=\frac{\phi_i^2\phi_j}{u_i}-\frac{\phi_i^3u_j}{u_i^2}$, where $D=(0,\infty)\times (0, \infty)$ and  $x_i=(\phi_i,u_i) \in D$. 
It can be verified that 
\begin{align*}
F_{ij}(x_i, x_j)+F_{ji}(x_j, x_i) =
& \, \left(\frac{\phi_i^2\phi_j}{u_i}-\frac{\phi_i^3u_j}{u_i^2}\right)
+ \left(\frac{\phi_i\phi_j^2}{u_j}-\frac{\phi_j^3u_i}{u_j^2}\right)\\
=&\,\, u_iu_j\frac{\phi_i^2}{u_i^2}\left(\frac{\phi_j}{u_j}-\frac{\phi_i}{u_i}\right)+u_iu_j\frac{\phi_j^2}{u_j^2}\left(\frac{\phi_i}{u_i}-\frac{\phi_j}{u_j}\right)\\
=&\, -u_iu_j\left(\frac{\phi_i}{u_i}-\frac{\phi_j}{u_j}\right)^2\left(\frac{\phi_i}{u_i}+\frac{\phi_j}{u_j}\right)\le 0,
\end{align*}
and the equal sign holds if and only if $\phi_i/u_i=\phi_j/u_j$.
Since $\mathcal{G}$ is cycle-balanced, it follows from Theorem~\ref{thm-sy} that 
\begin{equation}\label{claim}
\ds\sum_{i=1}^{n} \ds\sum_{j=1, j\neq i}^{n}\alpha_i L_{ij} F_{ij}(x_i,x_j)  = \ds\sum_{i=1}^{n} \ds\sum_{j=1, j\neq i}^{n}\alpha_i L_{ij} \left(\frac{\phi_i^2\phi_j}{u_i}-\frac{\phi_i^3u_j}{u_i^2}\right) \le 0,
\end{equation}
and the equal sign holds if and only if 
\begin{equation}\label{equ}
    \frac{\phi_1}{u_1}=\frac{\phi_2}{u_2}=\cdots=\frac{\phi_n}{u_n}.
\end{equation}

Similarly, by using \eqref{ss} and \eqref{eig}, we obtain
\begin{equation}\label{diff2}
    \mu_v \ds\sum_{i=1}^{n} \ds\sum_{j=1, j\neq i}^{n}\alpha_i L_{ij} \left(\frac{\psi_i^2\psi_j}{v_i}-\frac{\psi_i^3v_j}{v_i^2}\right)=\ds\sum_{i=1}^{n}\alpha_i \psi_i^2(b\phi_i+\psi_i)-\lambda\ds\sum_{i=1}^{n}\alpha_i \frac{\psi_i^3}{v_i}.
\end{equation}
Similar to \eqref{claim}, it follows from Theorem~\ref{thm-sy} that 
\begin{equation}\label{claim2}
 \ds\sum_{i=1}^{n}\ds\sum_{j=1, j\neq i}^{n}\alpha_i L_{ij} \left(\frac{\psi_i^2\psi_j}{v_i}-\frac{\psi_i^3v_j}{v_i^2}\right)\ge 0,
\end{equation}
where the equal sign holds if and only if
\begin{equation}\label{equ2}
    \frac{\psi_1}{v_1}=\frac{\psi_2}{v_2}=\cdots=\frac{\psi_n}{v_n}.
\end{equation}

Multiplying \eqref{diff2} by $c^3$ and subtracting it from \eqref{diff1}, and noticing \eqref{claim} and \eqref{claim2}, we have
\begin{eqnarray*}
-\lambda  \ds\sum_{i=1}^{n}\alpha_i \left(\frac{\phi_i^3}{u_i}- c^3\frac{\psi_i^3}{v_i}\right)&\le& -\ds\sum_{i=1}^{n}\alpha_i \phi_i^2(\phi_i+c\psi_i)+\ds\sum_{i=1}^{n}\alpha_i c^2 \psi_i^2(bc\phi_i+c\psi_i)\\
&\le& -\ds\sum_{i=1}^{n}\alpha_i \phi_i^2(\phi_i+c\psi_i)+\ds\sum_{i=1}^{n}\alpha_i c^2 \psi_i^2(\phi_i+c\psi_i)\\
&=&-\ds\sum_{i=1}^{n}\alpha_i (\phi_i-c\psi_i)(\phi_i+c\psi_i)^2
\le0,
\end{eqnarray*}
where we have used $bc\le 1$. This implies $\lambda\ge 0$ and the equality holds if and only if $bc=1$, $\phi_i=c\psi_i$ for all $i$, and \eqref{equ} and \eqref{equ2} hold. 

Now we consider the situation when $\lambda=0$. It follows that $u=kv$ for some $k>0$, and $bc=1$. By $u=kv$ and the first equation of \eqref{ss}, we have
$$w^*(\mu_u, p)=\left(1+\frac{c}{k}\right)u$$
By the second equation of \eqref{ss}, we have
$$
w^*(\mu_v, q)=\left(kb+1\right)v.
$$
Therefore,
$$
\frac{w^*(\mu_u, p)}{w^*(\mu_v,q)}=\frac{\left(1+\frac{c}{k}\right)u}{\left(kb+1\right)v}=c.
$$
\end{proof}

Denote by $\lambda_1(\mu,h)$ the principal eigenvalue of
\begin{equation*}
\mu\sum_{j=1}^nL_{ij}\psi_j+h_i\psi_i+\lambda\psi_i=0,\;\;i=1,\dots,n,
\end{equation*}
where $h=(h_1,\dots,h_n)$. Then $\lambda_1(\mu,h)=-s(\mu L+diag(h_i))$. Let $E_1=(w^*(\mu_u, p),\mathbf{0})$ and $E_2=(\mathbf{0},w^*(\mu_v,q))$ be the two semi-trivial equilibria of \eqref{main}. 
Following \cite{he2016global} we define the following parameter subsets of $Q=\{(\mu_u,\mu_v):\mu_u,\mu_v>0\}$.
\begin{equation*}
\begin{split}
&S_u=\{(\mu_u,\mu_v):(w^*(\mu_u, p),\mathbf{0})\text{ is linearly stable}\}=\{(\mu_u,\mu_v):\lambda_1(\mu_v,q-b w^*(\mu_u, p))>0\},\\
&S_v=\{(\mu_u,\mu_v):(\mathbf{0},w^*(\mu_v,q)) \text{ is linearly stable}\}=\{(\mu_u,\mu_v):\lambda_1(\mu_u,p-cw^*(\mu_v,q))>0\},\\
&S_{-}=\{(\mu_u,\mu_v):\lambda_1(\mu_v,q-b w^*(\mu_u, p))<0,\; \lambda_1(\mu_u,p-cw^*(\mu_v,q))<0\},\\
&S_{u,0}=\{(\mu_u,\mu_v):\lambda_1(\mu_v,q-bw^*(\mu_u,p))=0\},\\
&S_{v,0}=\{(\mu_u,\mu_v):\lambda_1(\mu_u,p-cw^*(\mu_v,q))=0\},\\
&S_{0,0}=\{(\mu_u,\mu_v):\lambda_1(\mu_v,q-bw^*(\mu_u,p))=\lambda_1(\mu_u,p-cw^*(\mu_v,q))=0\}.
\end{split}
\end{equation*}

We classify the global dynamics of \eqref{main1} according to the diffusion coefficients as follows.
\begin{theorem}\label{theorem-class}
Suppose that (A1)-(A3) hold. Then we have the following mutually disjoint decomposition of $Q$:
\begin{equation}\label{Qc}
Q=\left(S_u\cup S_{u,0}\setminus S_{0,0}\right)\bigcup\left(S_v\cup S_{v,0}\setminus S_{0,0}\right)\bigcup S_{-}\bigcup S_{0,0}.
\end{equation}
Moreover, the following statements hold for system \eqref{main1}:
\begin{enumerate}
\item [{\rm (i)}] For any $(\mu_u,\mu_v)\in S_u\cup S_{u,0}\setminus S_{0,0}$, $E_1=(w^*(\mu_u, p),\mathbf{0})$ is globally asymptotically stable.
\item [{\rm (ii)}] For any $(\mu_u,\mu_v)\in S_v\cup S_{v,0}\setminus S_{0,0}$, $E_2=(\mathbf{0},w^*(\mu_v,q))$ is globally asymptotically stable.
\item [{\rm (iii)}] For any $(\mu_u,\mu_v)\in S_{-}$,  \eqref{main1} has a unique positive equilibrium $(u,v)$, which is globally asymptotically stable.
\item [{\rm (iv)}] For any $(\mu_u,\mu_v)\in S_{0,0}$, we have $bc=1$, $w^*(\mu_u, p)\equiv cw^*(\mu_v,q)$ and  \eqref{main1} has a compact global attractor consisting of a continuum of equilibria 
\begin{equation}\label{co}
\{\left(\rho w^*(\mu_u, p),(1-\rho) w^*(\mu_u, p)/c\right):\rho\in [0,1]\}.
\end{equation}
\end{enumerate}
\end{theorem}

\begin{proof}
\underline{\emph{Step 1}} We first show the mutually disjoint decomposition of $Q$.  For simplicity of notations, we denote $u^*=w^*(\mu_u, p)$ and $v^*=w^*(\mu_v,q)$.
Denote by $\psi=(\psi_1,\dots,\psi_n)>\mathbf 0$ (respectively, $\phi=(\phi_1,\dots,\phi_n)>\mathbf 0$) the principal eigenvector with respect to $\la_1(\mu_v,q-bu^*)$ (respectively, $\la_1(\mu_u,p-cv^*)$).
Then we have
\begin{equation}\label{semistable}
\begin{cases}
\mu_u\ds\sum_{j=1}^{n}L_{ij}\phi_j+(p_i-cv^*_i)\phi_i+\lambda_1(\mu_u,p-cv^*)\phi_i=0, &i=1,\dots,n,\\
\mu_v\ds\sum_{j=1}^{n}L_{ij}\psi_j+(q_i-bu^*_i)\psi_i+\lambda_1(\mu_v,q-bu^*)\psi_i=0,&i=1,\dots,n.
\end{cases}
\end{equation}
Note that $u^*=(u^*_1,\cdots,u^*_n)$ and $v^*=(v^*_1,\cdots,v^*_n)$ satisfy 
\begin{equation}\label{semisteady}
\begin{cases}
\mu_u\ds\sum_{j=1}^{n}L_{ij}u^*_j+u_i(p_i-u^*_i)=0, &i=1,\dots,n,\\
\mu_v\ds\sum_{j=1}^{n}L_{ij}v^*_j+v_i(q_i-v^*_i)=0,&i=1,\dots,n.
\end{cases}
\end{equation}
Multiplying the first equation of \eqref{semistable} by
$u^*_i$, the first equation of \eqref{semisteady} by $\phi_i$ and taking the difference, we have
\begin{equation}
\label{diffsemi}
\mu_u \ds\sum_{j=1, j\neq i}^{n}L_{ij} (\phi_ju^*_i-\phi_i u^*_j)=(cv^*_i-u^*_i)u_i^*\phi_i-
\lambda_1(\mu_u,p-c v^*)\phi_i u^*_i, \;\; i=1,\dots,n
\end{equation}
Let $(\alpha_1, ..., \alpha_n)$ be the left eigenvector of $\mathcal{L}$ as defined in \eqref{eq-ev}. Multiplying both sides of \eqref{diffsemi} by $\alpha_i u_i^*/\phi_i$, and taking the sum, we obtain
\begin{equation}\label{diffsemi2}
\begin{split}
   &\mu_u \ds\sum_{i=1}^{n}\ \ds\sum_{j=1, j\neq i}^{n}\alpha_i L_{ij} \left(\frac{\phi_j(u^*_i)^2}{\phi_i}-u_i^*u^*_j\right)\\
   =&\ds\sum_{i=1}^{n}\alpha_i (cv^*_i-u^*_i)(u_i^*)^2-\lambda_1(\mu_u,p-c v^*)\ds\sum_{i=1}^{n}\alpha_i(u_i^*)^2.
\end{split}\end{equation}
Let $F_{ij}=F_{ij}(\phi_i,\phi_j) = \ds\frac{\phi_j(u^*_i)^2}{\phi_i}-u_i^*u^*_j$. Since \begin{eqnarray*}
F_{ij}+F_{ji}= \frac{\phi_j(u^*_i)^2}{\phi_i}-u_i^*u^*_j +
\frac{\phi_i(u^*_j)^2}{\phi_j}-u_j^*u^*_i
=\left(\frac{\sqrt{\phi_j}}{\sqrt{\phi_i}}u_i^*-\frac{\sqrt{\phi_i}}{\sqrt{\phi_j}}u_j^*\right)^2\ge0,
\end{eqnarray*}
and the equality holds if and only if $\phi_i/u^*_i=\phi_j/u^*_j$, it follows from Theorem~\ref{thm-sy} that \begin{equation}\label{semiclaim}
 \ds\sum_{i=1}^{n} \ds\sum_{j=1, j\neq i}^{n}\alpha_i L_{ij} \left(\frac{\phi_j(u^*_i)^2}{\phi_i}-u_i^*u^*_j\right)\ge0,
\end{equation}
and the equal sign holds if and only if
\begin{equation}\label{equsemi1}
    \frac{\phi_1}{u^*_1}=\frac{\phi_2}{u^*_2}=\cdots=\frac{\phi_n}{u^*_n}.
\end{equation}
Similarly, by using \eqref{semistable} and \eqref{semisteady}, we have
\begin{equation}\label{diffsemi2v}
\begin{split}
    &\mu_v \ds\sum_{i=1}^{n} \ds\sum_{j=1,j\neq i}^{n}\alpha_i L_{ij} \left(\frac{\psi_j(v^*_i)^2}{\psi_i}-v_i^*v^*_j\right)\\
    =&\ds\sum_{i=1}^{n}\alpha_i (bu^*_i-v^*_i)(v_i^*)^2-\lambda_1(\mu_v,q-b u^*)\ds\sum_{i=1}^{n}\alpha_i(v_i^*)^2,
\end{split}\end{equation}
and
\begin{equation}\label{semiclaimv}
 \ds\sum_{i=1}^{n} \ds\sum_{j=1, j\neq i}^{n}\alpha_i L_{ij} \left(\frac{\psi_j(v^*_i)^2}{v_i}-v_i^*v^*_j\right)\ge0,
\end{equation}
where the equal sign holds if and only if
\begin{equation}\label{equsemi}
    \frac{\psi_1}{v^*_1}=\frac{\psi_2}{v^*_2}=\cdots=\frac{\psi_n}{v^*_n}.
\end{equation}

Multiplying \eqref{diffsemi2v} by $c^3$ and subtracting it from \eqref{diffsemi2}, and noticing \eqref{semiclaim} and  \eqref{diffsemi2}, we have
\begin{equation}\label{twoequality}
\begin{split}
&\lambda_1(\mu_u,p-c v^*)\ds\sum_{i=1}^{n}\alpha_i(u_i^*)^2+\lambda_1(\mu_v,q-b u^*)c^3\ds\sum_{i=1}^{n}\alpha_i(v_i^*)^2\\
\le&\ds\sum_{i=1}^{n}\alpha_i (cv^*_i-u^*_i)(u_i^*)^2+\ds\sum_{i=1}^{n}\alpha_i (bcu^*_i-cv^*_i)(cv_i^*)^2\\
\le&-\sum_{i=1}^{n}\alpha_i(cv_i^*-u_i^*)^2(cv_i^*+u^*_i)\le0,
\end{split}
\end{equation}
where the last two inequalities are equalities if and only if $cv^*=u^*$ and $bc=1$.
This implies that
\begin{equation}\label{con1}
\left(S_u\cup S_{u,0}\setminus S_{0,0}\right)\cap\left(S_v\cup S_{v,0}\setminus S_{0,0}\right)=\emptyset,
\end{equation}
which proves \eqref{Qc}.

\noindent\underline{\emph{Step 2}} { We show that if $(\mu_u, \mu_v)\in S_u$ then $E_1$ is globally asymptotically stable. If $(\mu_u, \mu_v)\in S_u$, then $E_1$ is linearly asymptotically stable (so it is an isolated equilibrium) and $E_2$ is linearly unstable by \eqref{Qc}. By Lemma \ref{M1}, either $E_1$ is globally asymptotically stable or there exists a positive equilibrium $E\in [E_2, E_1]$. If such a positive equilibrium $E$ exists, it is stable by Theorem \ref{theorem_stable}. Then by Lemma \ref{M3}, there exists an unstable positive equilibrium in $[E, E_1]$, which contradicts Theorem \ref{theorem_stable}. This proves the global stability of $E_1$. Similarly, we can prove the case $(\mu_u, \mu_v)\in S_v$.}

\noindent\underline{\emph{Step 3}} As in \cite{he2016global}, we rule out the possibility of coexistence equilibrium in the following two cases:
\begin{enumerate}
    \item [(A)] $(\mu_u,\mu_v)\in S_{u,0}\setminus S_{0,0}$;
    \item [(B)] $(\mu_u,\mu_v)\in S_{v,0}\setminus S_{0,0}$.
\end{enumerate}
It suffices to consider case (A).
To the contrary, we assume that there exists a coexistence
equilibrium $(U^*,V^*)$ for some $(\mu_u,\mu_v)\in S_{u,0}\setminus S_{0,0}$ and $(b, c)=(b_0, c_0)$ satisfying $b_0c_0\leq 1$. So, $\lambda_1(\mu_v,q-b_0 u^*)=0$ and $\lambda_1(\mu_u,p-c_0 v^*)<0$. We define
$$
G(b, c, u, v)=
\begin{pmatrix}
\mu_u\ds\sum_{j=1}^{n}L_{1j}u_j+u_1(p_1-u_1-cv_1)\\
\vdots\\
\mu_u\ds\sum_{j=1}^{n}L_{nj}u_j+u_n(p_n-u_n-cv_n)\\
\mu_v\ds\sum_{j=1}^{n}L_{1j}v_j+v_1(q_1-bu_1-v_1)\\
\vdots\\
\mu_v\ds\sum_{j=1}^{n}L_{nj}v_j+v_n(q_n-bu_n-v_n)
\end{pmatrix}.
$$
Then we compute the Jacobian matrix of $G$ evaluated at $(b, c, u, v)=(b_0, c_0, U^*, V^*)$:
$$
DG_{(u, v)}(b_0, c_0, U^*, V^*)=
\begin{bmatrix}
\mu_u L+\text{diag}(p-2U^*-c_0V^*) & -\text{diag}(c_0U^*) \\
-\text{diag}(b_0V^*) & \mu_vL+\text{diag}(q-b_0U^*-2V^*)
\end{bmatrix}.
$$
By Theorem \ref{theorem_stable}, the principal eigenvalue of $DG_{(u, v)}(b_0, c_0, U^*, V^*)$ is negative and so all of its eigenvalues are on the {left} half plane of the complex plane. Therefore, $DG_{(u, v)}(b_0, c_0, U^*, V^*)$ is invertible. By $G(b_0, c_0, U^*, V^*)=0$ and the implicit function theorem, there exist positive solutions $(u(b,c), v(b,c))$ of $G(b, c, u, v)=0$ for $(b, c)$ close to $(b_0, c_0)$, where $(u(b,c), v(b,c))$ is continuously differentiable in $(b, c)$ with $(u(b_0, c_0), v(b_0, c_0))=(U^*, V^*)$. By the definition of $G$, $(u(b,c), v(b,c))$ is a positive equilibrium of \eqref{main1}. Noticing $\lambda_1(\mu_u,p-c_0 v^*)<0$, we may choose $(\check b, \check c)$ close to $(b_0, c_0)$ with $0<\check c<c_0$, $\check b>b_0$, and $\check b \check c\le 1$ such that \eqref{main1} has a positive equilibrium $(u(\check b, \check c), v(\check b, \check c))$ with $\lambda_1(\mu_u,p-\check c v^*)<0$. Since $\check b>b_0$, we have
$$
\lambda_1(\mu_v,q-\check b u^*)>\lambda_1(\mu_v,q-b_0 u^*)=0.
$$
 Then $(\check b, \check c)\in S_u$, which means $(u^*, 0)$ is globally asymptotically stable. This contradicts that \eqref{main1} has a positive equilibrium $(u(\check b, \check c), v(\check b, \check c))$. 

\noindent\underline{\emph{Step 4}} { If $(\mu_u,\mu_v)\in S_{u,0}\setminus S_{0,0}$, by Step 3 and Lemma \ref{M1}, $E_1$ is globally asymptotically stable; Similarly if $(\mu_u,\mu_v)\in S_{v,0}\setminus S_{0,0}$,  $E_2$ is globally asymptotically stable. 

Suppose $(\mu_u,\mu_v)\in S_-$. By Lemma \ref{M2}, there exist positive equilibiria $E^1$ and $E^2$ satisfying $E_2\ll_K E^2\le_K E^1\ll_K E_1$ such that the order interval $[E^2, E^1]$
attracts all solutions with initial data $(u_0, v_0)$ satisfying $u_0>0$ and $v_0>0$. By Theorem \ref{theorem_stable}, $E^1$ and $E^2$ are locally asymptotically stable. If $E^1\neq E^2$, by Lemma \ref{M3}, there exists an unstable equilibrium in $[E^2, E^1]$, which contradicts Theorem \ref{theorem_stable}. Therefore, $E^1=E^2$ and it is globally asymptotically stable. }


\noindent\underline{\emph{Step 5}} Finally, we show (iv). By \eqref{twoequality}, if $\lambda_1(\mu_u, p-cv^*)=\lambda_1(\mu_v, q-bu^*)= 0$ then $u^*=cv^*$ and $bc=1$, i.e.
$$S_{0,0}\subset\{(\mu_u,\mu_v):u^*=cv^*\text{ and }bc=1\}.$$
On the other hand, if $u^*=cv^*$ and $bc=1$,  from $\mu_u\ds\sum_{j=1}^n L_{ij}u_j^*+u_i^*(p_i-u_i^*)=0$, we have
$
\mu_u\ds\sum_{j=1}^n L_{ij} v_j^*+ v_i^*(p_i-c v_i^*)=0,
$
which implies  $\lambda_1(\mu_u, p-cv^*)= 0$. Similarly, $\lambda_1(\mu_v, q-bu^*)=0$. So we have 
$$
\{(\mu_u,\mu_v):u^*=cv^*\text{ and }bc=1\}\subset S_{0,0}.
$$
Hence,
\begin{equation}\label{con2}
S_{0,0}=\{(\mu_u,\mu_v):u^*=cv^*\text{ and }bc=1\}.
\end{equation}
It is easy to check that \eqref{main1} has a continuum of equilibria  \eqref{co} when $(\mu_u, \mu_v)\in S_{0, 0}$, which is a global attractor by
\cite[Theorem 3]{hirsch2006asymptotically} in the sense that every solution of \eqref{main1} converges to an equilibrium in \eqref{co}. The result in \cite[Theorem 3]{hirsch2006asymptotically} is for reaction-diffusion competition systems, which also holds for patch models.
\end{proof}

\section{{Examples}}
In this section, we apply the results obtained in Section 3 to two special situations {when the two species have the same resources availability or when the two species are identical except for resources availability}. { In these two cases, we are able to find the explicit parameter ranges that result in coexistence or competitive exclusion of the two species.}

The following result is needed later (see \cite{altenberg2012resolvent, chen2019spectral}).
\begin{lemma}\label{lemma_limit}
Suppose that $n\times n$ matrix $A$ is irreducible and quasi-positive (i.e. off-diagonal entries are nonnegative) with $s(A)=0$. Let $\eta=(\eta_1, \eta_2, ..., \eta_n)^T$ be a positive right eigenvector of $A$ corresponding to eigenvalue $s(A)=0$ and $D=\text{diag}(d_j)$ be a diagonal matrix.
Then, 
$$
\frac{d}{da} s(a A+D)\le 0, \ \text{for all } a>0,
$$
where the equality holds if and only if $D$ is a multiple of the identity matrix $I$. 
Moreover, the following limits hold:
$$
\lim_{a\rightarrow 0} s(a A+D)=\max_{1\le i\le n}\{d_i\}\ \ \ \text{and}\ \ \ \lim_{a\rightarrow \infty} s(a A+D)=\sum_{i=1}^n\eta_i d_i/\sum_{i=1}^n\eta_i.
$$
\end{lemma}

\subsection{Example (A)}
Firstly, we consider a situation that two species compete for the common resource (i.e. $p=q=r$):

\begin{equation}\label{main2}
\begin{cases}
u_i'=\mu_u\ds\sum_{j=1}^{n}L_{ij}u_j+u_i(r_i-u_i-cv_i), &i=1,\dots,n, t>0,\\
v_i'=\mu_v\ds\sum_{j=1}^{n}L_{ij}v_j+v_i(r_i-bu_i-v_i),&i=1,\dots,n, t>0, \\
u(0)=u_0\ge(\not\equiv)\mathbf0,\;v(0)=v_0\ge(\not\equiv)\mathbf0,
\end{cases}
\end{equation}
where $r>\mathbf 0$. We remark that the continuous space version of model \eqref{main2} has been investigated in \cite{zhou2018global}.
{ Then the following results can be derived from Theorem \ref{theorem-class}}.

\begin{proposition}\label{theoremA}
Suppose that $0<bc\le1$ and $(A2)-(A3)$ holds. 
Let $\theta=(\theta_1, \theta_2, ..., \theta_n)$  be a positive right eigenvector of $L$ corresponding to $s(L)=0$ with $\sum_{i=1}^n\theta_i=1$. Then the following statements hold for \eqref{main2}:
\begin{enumerate}
\item [{\rm (i)}] If $r=\delta\theta$ for some $\delta>0$, then we have:
\begin{enumerate}
\item [{\rm (i$_1$)}] If $(b,c)=(1,1)$, there exists a compact global attractor consisting of a continuum of equilibria 
$\{(\rho r,(1-\rho)r):\rho\in[0,1]\};$
\item [{\rm (i$_2$)}] If $b\ge 1$ and $c< 1$, $E_1=(w^*(\mu_u, r),\mathbf 0)$ is globally asymptotically stable;
\item [{\rm (i$_3$)}] If   $b< 1$ and $c\ge 1$, $E_2=(\mathbf 0,w^*(\mu_v, r))$ is globally asymptotically stable;
\item [{\rm (i$_4$)}] If $b<1$ and $c<1$, there exists a unique positive equilibrium 
$$
E=\ds\left(\ds\f{1-c}{1-bc}r,\ds\f{1-b}{1-bc}r\right),
$$
which is globally asymptotically stable.
\end{enumerate}
\item [{\rm (ii)}] If $r\ne \delta\theta$ for any $\delta>0$, then we have:
\begin{enumerate}
\item [{\rm (ii$_1$)}] Suppose $\mu_u<\mu_v$. Then there exist $b^*<1$ and $c^*>1$ with $b^*c^*>1$ such that  if  $b<b^*$ and $c<c^*$ \eqref{main2} has a unique positive equilibrium which is globally asymptotically stable; if $b\ge b^*$, $E_1=(w^*(\mu_u, r),\mathbf 0)$ is globally asymptotically stable; if $c\ge c^*$, $E_2=({\mathbf 0}, w^*(\mu_v, r))$ is globally asymptotically stable;

\item [{\rm (ii$_2$)}] Suppose $\mu_u>\mu_v$. Then there exist $b^*>1$ and $c^*<1$ with $b^*c^*>1$ such that  if  $b<b^*$ and $c<c^*$ \eqref{main2} has a unique positive equilibrium which is globally asymptotically stable; if $b\ge b^*$, $E_1=(w^*(\mu_u, r),\mathbf 0)$ is globally asymptotically stable; if $c\ge c^*$, $E_2=({\mathbf 0}, w^*(\mu_v, r))$ is globally asymptotically stable;

\item [{\rm (ii$_3$)}] Suppose $\mu_u=\mu_v$. Then if $b<1$ and $c<1$, \eqref{main2} has a unique positive equilibrium which is globally asymptotically stable; if $b\ge 1$ and $c< 1$,  $E_1=(w^*(\mu_u, r),\mathbf 0)$ is globally asymptotically stable; if $c\ge 1$ and $b<1$,  $E_2=(\mathbf 0,w^*(\mu_v, r))$ is globally asymptotically stable; if $(b,c)=(1, 1)$, there exists a compact global attractor consisting of a continuum of steady states
$
\{(\rho w^*(\mu_u, r),(1-\rho)w^*(\mu_v, r)):\rho\in[0,1]\}.
$
 
\end{enumerate}
\end{enumerate}
\end{proposition}
\begin{proof}
 (i) If $r=\delta\theta$ for some $\delta>0$, we have
\begin{equation*}
w^*(\mu_u, r)=w^*(\mu_v,r)=r.
\end{equation*}
For the case $b=c=1$, a direct computation yields 
\begin{equation*}
\begin{split}
&\lambda_1(\mu_v,r-b w^*(\mu_u, r))=\lambda_1(\mu_v,\mathbf 0)=0,\\
&\lambda_1(\mu_u,r-c w^*(\mu_v, r))=\lambda_1(\mu_u,\mathbf 0)=0.
\end{split}
\end{equation*}
Then it follows from Theorem \ref{theorem-class} (iv) that  (i$_1$) holds.

For the case $b\ge 1$ and $c<1$,  we have
 \begin{equation*}
\begin{split}
&\lambda_1(\mu_u,r-b w^*(\mu_v, r))=\lambda_1(\mu_u,(1-b)r)\ge 0,\\
&\lambda_1(\mu_u,r-c w^*(\mu_v, r))=\lambda_1(\mu_u,(1-c)r)<0.
\end{split}
\end{equation*}
Then it follows from Theorem \ref{theorem-class} (i) that $E_1=(w^*(\mu_u, r),\mathbf 0)$ is globally asymptotically stable, which proves (i$_2$).
Similarly, we can prove (i$_3$).

For the case $b<1$ and $c<1$, we have
 \begin{equation*}
\begin{split}
&\lambda_1(\mu_v,r-bw^*(\mu_u, r))=\lambda_1(\mu_v,(1-b)r)<0,\\
&\lambda_1(\mu_u,r-c w^*(\mu_v, r))=\lambda_1(\mu_u,(1-c)r)<0.
\end{split}
\end{equation*}
Then it follows from  Theorem \ref{theorem-class} (iii) that statement (i$_4$) holds.

 (ii) Suppose $r\ne \delta \theta$ for any $\delta>0$. Then $r-w^*(\mu_u, r)\ne \gamma (1,\dots,1) $ for any $\gamma\in \mathbb R$. If this is not true,
then there exists $\gamma_0\in\mathbb R$ such that $r-w^*(\mu_u, r)= \gamma_0 (1,\dots,1) $. Since $w^*:=w^*(\mu_u, r)$ satisfies 
$$
\mu_u \sum_{j=1}^n L_{ij}w^*_j+w^*_i(r_i-w^*_i)=0,
$$
we conclude that $-\gamma_0/\mu_u$ is the principal eigenvalue of $L$ with $w^*(\mu_u, r)$ being a positive eigenvector. Then $\gamma_0=0$, and $r=w^*(\mu_u, r)=\delta_0\theta$ for some $\delta_0>0$, which is a contradiction. Similarly, we obtain that $r-w^*(\mu_v, r)\ne \gamma (1,\dots,1) $ for any $\gamma\in \mathbb R$.
Then it follows from Lemma \ref{lemma_limit} that 
$\la_1\left(\mu, r-w^*(\mu_u, r)\right)$ and $\la_1\left(\mu, r-w^*(\mu_v, r)\right)$ {are} strictly increasing for $\mu\in(0,\infty)$.

Note that 
\begin{equation*}
\la_1\left(\mu_u, r-w^*(\mu_u, r)\right)=0\;\; \text{and}\;\;\la_1\left(\mu_v, r-w^*(\mu_v, r)\right)=0.
\end{equation*}
(ii$_1$) If $\mu_u<\mu_v$, we have
\begin{equation*}
\la_1\left(\mu_v, r-bw^*(\mu_u, r)\right)|_{b=1}>0\;\; \text{and}\;\;\la_1\left(\mu_u, r-cw^*(\mu_v, r)\right)|_{c=1}<0.
\end{equation*}
Since $\la_1\left(\mu_v, r-bw^*(\mu_u, r)\right)|_{b=0}<0$ and $\la_1\left(\mu_v, r-bw^*(\mu_u, r)\right)$ is strictly increasing in $b$,  there exists $b^*\in(0,1)$ such that 
\begin{equation}\label{bstr}
\la_1\left(\mu_v, r-bw^*(\mu_u, r)\right)\  \begin{cases}
<0,&b<b^*,\\
=0,&b=b^*\\
>0,&b>b^*.
\end{cases}
\end{equation}
Since $\la_1\left(\mu_u, r-cw^*(\mu_v, r)\right)$ is strictly increasing in $c$ and $\lim_{c\rightarrow\infty}\la_1\left(\mu_u, r-cw^*(\mu_v, r)\right)=\infty$, there exists $c^*>1$ such that 
\begin{equation}\label{cstr}
\la_1\left(\mu_u, r-cw^*(\mu_v, r)\right)\ 
\begin{cases}
<0,&c<c^*,\\
=0,&c=c^*\\
>0,&c>c^*.
\end{cases}
\end{equation}

We claim  $b^*c^*>1$. To see it, we first note that $b^*c^*<1$ is not possible. If otherwise, we may find $(b, c)$ such that $b<b^*$ and $c<c^*$ with $bc<1$. For such $(b, c)$, we have $\la_1(\mu_v, r-bw^*(\mu_u, r))<0$ and $\la_1(\mu_u, r-cw^*(\mu_v, r))<0$, i.e. both $E_1$ and $E_2$ are stable, which is impossible by Theorem \ref{theorem-class}. 

Suppose to the contrary that $b^*c^*=1$. Since $\la_1(\mu_v, r-b^*w^*(\mu_u, r))=\la_1(\mu_u, r-c^*w^*(\mu_v, r))=0$, we have $w^*(\mu_u, r)=c^*w^*(\mu_v, r)$ by Theorem \ref{theorem-class} (iv). Putting this into 
\begin{eqnarray*}
\mu_u \sum_{j=1}^n L_{ij}w^*_j(\mu_u, r)+w^*_i(\mu_u, r)(r_i-w^*_i(\mu_u, r))=0,\\
\mu_v \sum_{j=1}^n L_{ij}w^*_j(\mu_v, r)+w^*_i(\mu_v, r)(r_i-w^*_i(\mu_v, r))=0,
\end{eqnarray*}
we obtain 
$$
w^*_i(\mu_v, r)=\frac{\mu_v-\mu_u}{c^*\mu_v-\mu_u}r_i<r_i \ \ \text{for all} \ i=1,...,n.
$$  
Therefore, noticing $\sum_{i=1}^n L_{ij}=0$ for all $j=1,...,n$, we have
$$0=\sum_{i=1}^n (\mu_v \sum_{j=1}^n L_{ij}w^*_j(\mu_v, r)+w^*_i(\mu_v, r)(r_i-w^*_i(\mu_v, r)))=\sum_{i=1}^n w^*_i(\mu_v, r)(r_i-w^*_i(\mu_v, r))>0,
$$
which is a contradiction. This proves the claim. It follows from \eqref{bstr}-\eqref{cstr}, $b^*c^*>1$ and Theorem \ref{theorem-class}  that (ii$_1$) holds. 

Using similar arguments for (ii$_1$), we can prove (ii$_2$). For the case (ii$_3$), we have $b^*=c^*=1$ and its proof is similar to (ii$_1$). 
\end{proof}

\begin{figure}[htbp]
 \centering
  \includegraphics[width=0.6\textwidth]{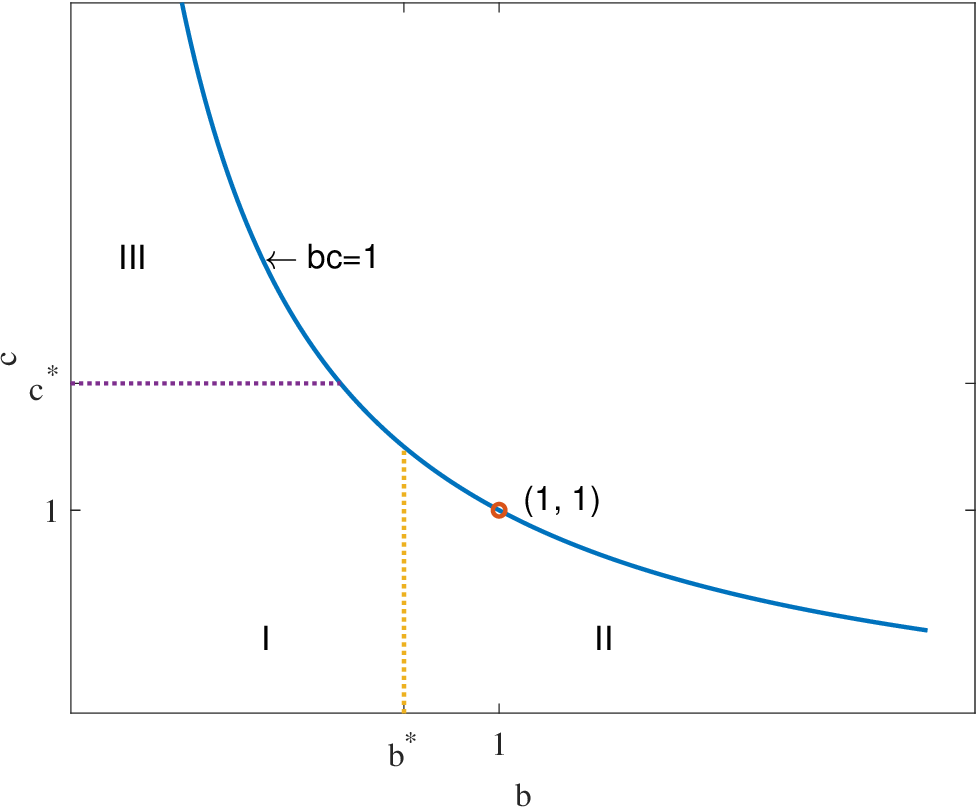}
 \caption{Illustration of {Proposition} \ref{theoremA} \rm (ii$_1$)  ($r\neq \delta\theta$ and $\mu_u<\mu_v$) for Example A. Here, $b^*<1$, $c^*>1$, and $b^*c^*>1$. In {regime} I,  there exists a globally asymptotically stable positive equilibrium; in {regime}  II, $E_1$ is globally asymptotically stable; in {regime}  III, $E_2$ is globally asymptotically stable. }
 \label{FigA}
\end{figure}


\begin{remark}
We have a complete classification of the global stability of Example A by {Proposition} \ref{theoremA}. In Figure \ref{FigA}, we plot a diagram to illustrate our results for the case \rm (ii$_1$) (i.e. $r\neq \delta\theta$ and $\mu_u<\mu_v$).  
\end{remark}

\subsection{Example (B)}

In this subsection, we consider a special case of \eqref{main1}:
\begin{equation}\label{main3}
\begin{cases}
u_i'=\mu\ds\sum_{j=1}^{n}L_{ij}u_j+u_i(p_i-u_i-v_i), &i=1,\dots,n, t>0,\\
v_i'=\mu\ds\sum_{j=1}^{n}L_{ij}v_j+v_i(q_i-u_i-v_i),&i=1,\dots,n, t>0, \\
u(0)=u_0\ge(\not\equiv)\mathbf 0,\;v(0)=v_0\ge(\not\equiv)\mathbf0.
\end{cases}
\end{equation}
Here the two species have the same intraspecific competition coefficients and diffusion rates, but different resources availability. 
The case $n=2$ (two-patch model) of \eqref{main3} has been investigated by \cite{cheng2019coexistence,gourley2005two,lin2014global}. In \cite{gourley2005two}, it was conjectured that if $q_1-\sigma=p_1<q_1<q_2<p_2=q_2+\sigma$ with $0<\sigma<q_1$, then $u_1$ out-competes $u_2$ when the dispersal rate $\mu$ is large while $u_1$ and $u_2$ coexist when $\mu$ is small. This conjecture was verified in \cite{lin2014global}, where the authors proved that a threshold value of $\mu^*$ dividing the outcome of the competition exists: if $\mu<\mu^*$ the model has a unique globally asymptotically stable positive equilibrium while if $\mu>\mu^*$ the $u_1$-only semitrivial equilibrium is globally asymptotically stable. In \cite{cheng2019coexistence}, it was shown that such a threshold result no longer holds when the inequalities on the birth rates $p_i, q_i$ are relaxed. We consider the general case with $n\ge 2$ here.

The following result states that the species with more resources  wins the competition:
\begin{proposition}
Suppose that (A1)-(A3) hold. If $p> q$, then $E_1=(w^*(\mu, p),\mathbf 0)$ is globally asymptotically stable for \eqref{main3}; if $p<q$, then $E_2=(\mathbf 0,w^*(\mu, q))$ is globally asymptotically stable  for \eqref{main3}.
\end{proposition}
\begin{proof}
We only show the case of $p> q$.  For simplicity of notations, we denote $u^*=w^*(\mu, p)$ and $v^*=w^*(\mu,q)$. 
By Theorem \ref{theorem-class}, it suffices to consider the signs of $\lambda_1(\mu, q-u^*)$ and $\lambda_1(\mu, p-v^*)$, where $u^*$ and $v^*$ also depend on $\mu$. Let $\phi$ be a positive eigenvector corresponding to  $\lambda_1(\mu, q-u^*)$. Then
\begin{equation}
\mu\sum_{j=1}^n L_{ij}\phi_j+(p_i-u^*_i)\phi_i+(q_i-p_i)\phi_i+\lambda_1(\mu, q-u^*)\phi_i=0.
\end{equation}
Denote $A=\mu L+\text{diag}(p_j-u^*_j)$ and $D=\text{diag}(q_j-p_j)$. Clearly, $u^*$ is a positive right eigenvector of $A$ corresponding with eigenvalue $0$. So $s(A)=0$. Therefore, by Lemma \ref{lemma_limit}, $s(a A+D)$ is strictly decreasing in $a$ and 
$$
\lim_{a\rightarrow 0} s(a A+D)=\max_{1\le j\le n}\{q_j-p_j\}\le0 \ \  \text{and} \ \  \lim_{a\rightarrow \infty} s(a A+D)=\sum_{i=1}^n\eta_i (q_i-p_i)<0,
$$
where $\eta=(\eta_1,\dots,\eta_n)$  is the positive right eigenvector of $A$ corresponding to $s(A)$ with $\sum_{i=1}^n\eta_i=1$.
This yields 
\begin{equation}\label{estiu}\lambda_1(\mu, q-u^*)=-s(A+D)>0.
\end{equation}
Similarly, we can prove that $\lambda_1(\mu, p-v^*)<0.$
 This, combined with \eqref{estiu}, implies that
 $(\mu,\mu)\in S_u\cup S_{u,0}\setminus S_{0,0},$
 and consequently,  $E_1$ is globally asymptotically stable by Theorem \ref{theorem-class}.
\end{proof}

Next we consider the case when the resources of the two competing species are not comparable. 

\begin{proposition}\label{theorem_limit}
Suppose that (A1)-(A3) hold, and  $p\not\ge q$ and $q\not\ge p$. Let $\theta=(\theta_1, \theta_2, ..., \theta_n)$  be a positive right eigenvector of $L$ corresponding to $s(L)=0$ satisfying $\sum_{i=1}^n\theta_i=1$. Then the following statements hold for \eqref{main3}:
\begin{enumerate}
    \item There exists $\mu_1>0$ such that  \eqref{main3} has a unique positive equilibrium which is globally asymptotically stable for $0<\mu<\mu_1$.
    \item If $\ds\sum_{j=1}^n\theta_j(p_j-q_j)>0$, then there exists $\mu_2>\mu_1$ such that $E_1$ is globally asymptotically stable for $\mu>\mu_2$; on the other hand if $\ds\sum_{j=1}^n\theta_j(p_j-q_j)<0$, then there exists $\mu_3>\mu_1$ such that $E_2$ is globally asymptotically stable for $\mu>\mu_3$.
\end{enumerate}
\end{proposition}
\begin{proof}
We claim that there exist { two positive numbers $m$ and $M$} such that $m \theta\le  u^* \le M \theta$ for all $\mu>0$. To see that, we may choose $m, M>0$  such that $u=m \theta$ is a lower solution and $u=M \theta$ is an upper solution of
\begin{equation}\label{ustart}
u_i'=\mu\ds\sum_{j=1}^{n}L_{ij}u_j+u_i(p_i-u_i), \ \ i=1,\dots,n.
\end{equation}
Since $u^*$ is the unique globally asymptotically stable positive equilibrium of \eqref{ustart}, we have $m \theta\le  u^* \le M \theta$ for all $\mu>0$.

1. Since $u^*$ is the unique positive equilibrium of \eqref{ustart},
we have $u^*_i\rightarrow p_i$ as $\mu\rightarrow 0$. This can be seen by taking $\mu\rightarrow 0$ in \eqref{ustart} and noticing $u^*\ge m \theta$.
By Lemma \ref{lemma_limit} and $\lim_{\mu\rightarrow 0} u^*=p$, we have $\ds\lim_{\mu\rightarrow 0}\lambda_1(\mu, q-u^*)=-\max_{1\le i\le n}\{q_i-p_i\}<0$. Similarly, $\ds\lim_{\mu\rightarrow 0}\lambda_1(\mu, p-v^*)=-\max_{1\le i\le n}\{p_i-q_i\}<0$. Therefore, there exists $\mu_1>0$ such that $\lambda_1(\mu, q-u^*), \lambda_1(\mu, p-v^*)<0$ for all $0<\mu<\mu_1$. Therefore by Theorem \ref{theorem-class}, \eqref{main3} has a unique positive equilibrium which is globally asymptotically stable for $0<\mu<\mu_1$.

2. We first claim:
\begin{equation}\label{uinf}
   \lim_{\mu\rightarrow\infty} u^*= \frac{\sum_{i=1}^n\theta_i p_i}{\sum_{i=1}^n\theta_i^2}\theta.
\end{equation}
To see that, for any $\mu_k\rightarrow\infty$, there exists a subsequence,  still denoting by itself, such that the corresponding positive solution $u_k^*=(u_{1k},..., u_{nk})$ of \eqref{ustart} satisfies $ u_k^*\rightarrow u^\infty\ge 0$ as $k\rightarrow\infty$. Dividing both sides of \eqref{ustart} by $\mu_k$ and taking $k\rightarrow\infty$, we obtain that $\ds\sum_{j=1}^n L_{ij}u^{\infty}_j=0$, which means that $u^{\infty}=l\theta$ for some $l\ge 0$. Summing up all the $n$ equations in \eqref{ustart} and noticing $\ds\sum_{i=1}^n L_{ij}=0$ for all $1\le j \le n$, we have $\ds\sum_{i=1}^n u_{ik}(p_i-u_{ik})=0$. Taking $k\rightarrow\infty$, we have $\ds\sum_{i=1}^n l\theta_i(p_i-l\theta_i)=0$. Since $l>0$ by $u^*\ge m \theta$, we have $l= {\sum_{i=1}^n\theta_i p_i}/{\sum_{i=1}^n\theta_i^2}$. This proves \eqref{uinf}.

Then by Lemma \ref{lemma_limit} and \eqref{uinf}, we have
$\ds\lim_{\mu\rightarrow\infty} \lambda_1(\mu, q-u^*)={\sum_{i=1}^n\theta_i (p_i-q_i)}.$
Similarly,
$\ds\lim_{\mu\rightarrow\infty} \lambda_1(\mu, p-v^*)={\sum_{i=1}^n\theta_i (q_i-p_i)}.
$
If $\ds\sum_{j=1}^n\theta_j(p_j-q_j)>0$, there exists $\mu_2>0$ such that $\lambda_1(\mu, q-u^*)>0$ and $ \lambda_1(\mu, p-v^*)<0$ for all $\mu>\mu_2$. Therefore, by Theorem \ref{theorem-class}, $E_1$ is globally asymptotically stable for $\mu>\mu_2$. The case $\ds\sum_{j=1}^n\theta_j(p_j-q_j)<0$ can be proved similarly.
\end{proof}

\begin{remark}
By {{Proposition}} \ref{theorem_limit}, one may expect that there exists a critical value $\mu^*$ such that if $\mu<\mu^*$, \eqref{main3} has a unique globally stable positive equilibrium while if $\mu>\mu^*$ either $E_1$ or $E_2$ is globally asymptotically stable for \eqref{main3}. However, this result does not hold in general. Indeed, in \cite{cheng2019coexistence}, it has {been} shown that for $n=2$ there might exist $0<\mu_1^*<\mu_2^*<\mu^*_3$ such that \eqref{main3} has a globally asymptotically stable positive equilibrium exactly when $\mu\in (0, \mu_1^*)\cup(\mu_2^*, \mu_3^*)$.
\end{remark}

The following result characterizes the asymptotic limit of the unique positive equilibrium of \eqref{main3} when the diffusion rate $\mu$ approaches zero.
\begin{proposition}
Suppose that (A1)-(A3) hold. Let $\Omega_u=\{i: 1\le i\le n, p_i>q_i\}$ and $\Omega_v=\{i: 1\le i\le n, p_i<q_i\}$. Suppose that $\Omega_u$ and $\Omega_v$ are not empty with $\Omega_u\cup\Omega_v=\{1, 2, ..., n\}$.
Let $u_0=(u_{01}, u_{02}, ..., u_{0n})$ and $v_0=(v_{01}, v_{02}, ..., v_{0n})$, where
\begin{equation*}
    u_{0i}=
\left\{
\begin{array}{cc}
  p_i,   & \text{if} \ i\in \Omega_u,  \\
    0, & \text{if} \ i\in \Omega_u,
\end{array}
\right.
\ \ \ \ \text{and}\ \ \ \
    v_{0i}=
\left\{
\begin{array}{cc}
  0   & \text{if} \ i\in \Omega_u,  \\
    q_i & \text{if} \ i\in \Omega_u.
\end{array}
\right.
\end{equation*}
Let $(u, v)$ be the unique positive equilibrium of \eqref{main3} when $\mu$ is small, then $\ds\lim_{\mu\rightarrow 0}(u, v)=(u_0, v_0)$.
\end{proposition}
\begin{proof}
By {Proposition} \ref{theorem_limit}, there exists $\mu_1>0$ such that \eqref{main3} has a unique positive equilibrium $(u, v)$ that is globally asymptotically stable. By the definition, when $\mu=0$, $(u_0, v_0)$ is a solution of the following system
\begin{equation}\label{sss}
\begin{cases}
\mu\ds\sum_{j=1}^{n}L_{ij}u_j+u_i(p_i-u_i-v_i)=0, &i=1,\dots,n,\\
\mu\ds\sum_{j=1}^{n}L_{ij}v_j+v_i(q_i-u_i-v_i)=0,&i=1,\dots,n.
\end{cases}
\end{equation}

We show that \eqref{sss} has a continuum of solutions emanating $(\mu, (u, v))=(0, (u_0, v_0))$. To see this, we define
$$
F(\mu, (u, v))=
\begin{pmatrix}
\mu\ds\sum_{j=1}^{n}L_{1j}u_j+u_1(p_1-u_1-v_1)\\
\vdots\\
\mu\ds\sum_{j=1}^{n}L_{nj}u_j+u_n(p_n-u_n-v_n)\\
\mu\ds\sum_{j=1}^{n}L_{1j}v_j+v_1(q_1-u_1-v_1)\\
\vdots\\
\mu\ds\sum_{j=1}^{n}L_{nj}v_j+v_n(q_n-u_n-v_n)
\end{pmatrix}.
$$
Then we compute the Jacobian matrix of $F$ evaluated at $(\mu, (u, v))=(0, (u_0, v_0))$:
$$
DF_{(u, v)}(0, (u_0, v_0))=
\begin{bmatrix}
\text{diag}(p-2u_0-v_0) & -\text{diag}(u_0) \\
-\text{diag}(v_0) & \text{diag}(q-u_0-2v_0)
\end{bmatrix}.
$$
By the assumption $p_i\neq q_i$ for all $i$ and the definition of $u_0$ and $v_0$, we can see that $DF_{(u, v)}(0, (u_0, v_0))$ is invertible. Therefore, by the implicit function theorem, there exists $\mu^*_1>0$ such that \eqref{sss} has a solution $(u(\mu), v(\mu))$ for each $0\le \mu< \mu^*_1$, where $(u(\mu), v(\mu))$ is continuous in $\mu$. By the definition of $(u_0, v_0)$, we may choose $\mu^*_1$ small such that $u_i(\mu)>0$ for each $i\in\Omega_u$ and $v_i(\mu)>0$ for each $i\in\Omega_v$ for all $0\le \mu \le \mu^*_1$.

We show that $u_i(\mu)>0$ for each $i\in\Omega_v$ and $v_i(\mu)>0$ for each $i\in\Omega_u$ for $\mu$ close to zero. To see that, fix $i_0\in \Omega_v$. Then, $p_{i_0}<q_{i_0}$, $u_{0i_0}=0$ and $v_{0i_0}=q_{i_0}$.  Differentiating $\mu\ds\sum_{j=1}^{n}L_{i_0j}u_j+u_{i_0}(p_{i_0}-u_{i_0}-v_{i_0})=0$ with respect to $\mu$ and evaluating at $(\mu, (u, v))=(0, (u_0, v_0))$, we obtain
$$
u'_{i_0}(0)=\frac{\sum_{j=1}^{n}L_{i_0j}u_{0j}}{q_{i_0}-p_{i_0}}= \frac{\sum_{j\in \Omega_u}L_{i_0j}u_{0j}}{q_{i_0}-p_{i_0}}.
$$
By the assumption, $\ds\sum_{j\in \Omega_u}L_{i_0j}>0$. So $u'_{i_0}(0)>0$. Therefore, $u_{i_0}(\mu)\approx u_{i_0}(0)+u'_{i_0}(0)\mu>0$ for $\mu$ close to zero. Since $i_0\in \Omega_v$ was arbitrary, $u_i(\mu)>0$ for each $i\in\Omega_v$ when $\mu$ is close to zero. Similarly,  $v_i(\mu)>0$ for each $i\in\Omega_u$ when $\mu$ is close to zero. 

We can find $\mu^*<\mu^*_1$ such that the solution $(u(\mu), v(\mu))$ of \eqref{sss} is positive for  $0< \mu <\mu^*$. Then the conclusion follows from the uniqueness of the positive solution of \eqref{sss} and the continuity of $(u(\mu), v(\mu))$ in $\mu$.
\end{proof}

\section{Conclusion}
{ In this paper, we analyze the global dynamics of a Lotka-Volterra competition model in patchy environment with asymmetric dispersal. Under the assumption of weak competition and the weighted digraph of the connection matrix is strongly connected and cycle-balanced, we classify the global dynamics of the model. In particular, in Theorem \ref{theorem-class}, we show that either the model has a globally stable coexistence steady state or one species competitively exclude the other one except for the special case that  both semi-trivial equilibria are neutrally stable. Theorem \ref{theorem-class} has been applied to two special cases, in which we are able to determine the explicit parameter ranges for coexistence verse competitive exclusion.}

{ Our results use techniques from two fields. We use matrix theory and graph theory techniques to deal with the asymmetry of the connection matrix $L$. Due to the limitation of this method, the weighted digraph of $L$ needs to be cycle-balanced. Though cycle-balanced digraph covers some important types of the configurations of patches,  it is desirable to see whether this condition can be removed. We conjecture that Theorem \ref{theorem-class}  still holds without this technical assumption. 

The second technique that we rely on is the well-developed monotone dynamical system theory. According to this theory, the dynamics of the model is essentially determined by the local dynamics of the equilibria. An essential step in our analysis is to prove that  every coexistence steady state is locally asymptotically stable except for the special case that both semi-trivial equilibria are neutrally stable. This approach has been adopted in two recent articles \cite{he2016global,zhou2018global} on Lotka-Volterra reaction-diffusion competition models. We remark that the patch model may not be simply regarded as the discretization of the reaction-diffusion model as the connection matrix $L$ is not assumed to be symmetric. 

Finally, we want to point out that we only consider the weak competition case in the paper, and the dynamics of the strong competition case remains an open problem. }

\section*{Appendix}

\begin{proof}[\underline{Proof of Proposition~\ref{tree}}]
1. If $A$ is symmetric, then $\mathcal{G}$ is clearly cycle-balanced from the definition. Now we assume that ${\mathcal G}$ is cycle-balanced. Suppose $a_{ij}>0$; that is, there is an arc $(j,i)$ from vertex $j$ to vertex $i$ in $\mathcal{G}$. Since  $\mathcal{G}$ is strongly connected, there exists a path from $i$ to $j$. Therefore, the arc $(j,i)$ belongs to some cycle $\mathcal{C}$. If  $\mathcal{G}$ is cycle-balanced, then for any cycle $\mathcal{C}$, its reverse $-\mathcal{C}$ is also a cycle in $\mathcal{G}$. This implies $a_{ji}>0$.   Hence $A$ must be sign pattern symmetric.  
 
2. Since $\mathcal{G}$ has $n$ vertices and it is strongly connected, any of its spanning tree has $n-1$ arcs. From part 1, any reverse arc is also an arc of $\mathcal{G}$, thus  $\mathcal{G}$ has at least $2(n-1)$ arcs. 

3. Let $A$ be a sign pattern symmetric $n\times n$  matrix with exactly $2(n-1)$ positive entries, and assume $A$ is irreducible. Let $A^+=(a_{ij}^+)_{n\times n}$ be defined by $a_{ij}^+=a_{ij}$ when $i>j$ and $a_{ij}^+=0$ when $i\leq j$, then the subdigraph $\mathcal{G}^+$ associated with $A^+$ is a tree. Similarly let $A^-=(a_{ij}^{-})_{n\times n}$ be defined by $a_{ij}^{-}=a_{ij}$ when $i<j$ and $a_{ij}^{-}=0$ when $i\geq j$, then the subdigraph $\mathcal{G}^{-}$ associated with $A^{-}$ is also a tree. The digraph $\mathcal{G}$ is the union of two disjoint trees $\mathcal{G}^+$ and $\mathcal{G}^-$. It is easy to see every cycle of $\mathcal{G}$ has exactly two vertices, and $\mathcal{G}$ is cycle-balanced as every $2$-cycle is naturally balanced. This proves such a bi-directional tree is cycle-balanced. If $n=2$ then every cycle of $\mathcal{G}$ has two vertices. Hence it must be cycle-balanced.

Now assume $\mathcal{G}$ is strongly connected, every cycle of $\mathcal{G}$ has exactly two vertices and $\mathcal{G}$ is cycle-balanced. From part 1, $A$ is sign pattern symmetric. So we only need to prove $\mathcal{G}$ has exactly $2(n-1)$ arcs. Let $\mathcal{T}$ be a spanning tree of $\mathcal{G}$, then $\mathcal{T}$ has $n-1$ arcs. Reversing the directions of all arcs in $\mathcal{T}$ yields $-\mathcal{T}$, which is also a subdigraph of $\mathcal{G}$. This implies that  $\mathcal{G}$ has at least $2(n-1)$ arcs, $n-1$ arcs in $\mathcal{T}$ and $n-1$ arcs in $-\mathcal{T}$. If in addition to the $2(n-1)$ arcs in the spanning tree $\mathcal{T}$ and its reverse $-\mathcal{T}$, there exists at least one more arc, say $(i_a,i_b)$, which is not in $\mathcal{T}\cup (-\mathcal{T})$. But there is a path $P$ from $i_b$ to $i_a$ in $\mathcal{T}\cup (-\mathcal{T})$ because of the property of spanning tree. The length of $P$ is at least $2$ as $(i_a,i_b)\not\in \mathcal{T}\cup (-\mathcal{T})$, so the union of $P$ and $(i_a,i_b)$ is a cycle with length at least $3$, which contradicts with the assumption that every cycle of $\mathcal{G}$ has exactly two vertices. Therefore $\mathcal{G}$ has exactly $2(n-1)$ arcs.

4. Assume that $\mathcal{G}$ is complete with at least $3$ vertices. If $\mathcal{G}$ is cycle-balanced, it is obvious that each $3$-cycle is balanced. So we only need to prove that if each  $3$-cycle is balanced, then each $k$-cycle with $k\geq 4$ is also balanced. We prove it inductively in $k$. When $k=3$, it is true from the assumption. Suppose it is true for any $k$-cycle with $k\leq m$, we show it is true for $k=m+1$. Let $\mathcal{C}$ be a cycle with length $m+1$. Without loss of generality, we assume that $\mathcal{C}=(1,2,\cdots,m,m+1,1)$, namely, a cycle connecting vertices $1, 2, \cdots, m, m+1, 1$ consecutively. Since  $\mathcal{G}$ is complete, we have $a_{m1}>0$ and $a_{1m}>0$. From the inductive hypothesis, the $m$-cycle ${\mathcal C}_1=(1,2,\cdots,m-1,m,1)$ and the $3$-cycle ${\mathcal C}_2=(1,m,m+1,1)$ are both balanced. Hence 
$a_{m1}a_{1m}w(\mathcal{C})=w(\mathcal{C}_1)w(\mathcal{C}_2)=w(-\mathcal{C}_1)w(-\mathcal{C}_2)=a_{m1}a_{1m}w(-\mathcal{C})$, which implies that $w(\mathcal{C})=w(-\mathcal{C})$. 
\end{proof}

\bibliographystyle{abbrv}
\bibliography{references,comp}

\end{document}